\documentclass{amsart}
\usepackage{amsfonts, amsbsy, amsmath, amssymb}
\usepackage{tikz}
\usepackage{mmacells}

\newtheorem{thm}{Theorem}[section]
\newtheorem{lem}[thm]{Lemma}

\newtheorem{algo}[thm]{Algorithm}

\newtheorem{thm-con}[thm]{Theorem-Conjecture}
\numberwithin{equation}{section}

\theoremstyle{definition}

\allowdisplaybreaks

\newcommand\longequation[1]{\parbox[t]{\textwidth}{\raggedright$#1$}}

\newcommand{\f}{\Bbb F}

\begin{document}

\title[An approach to normal polynomials]{An approach to normal polynomials through symmetrization and symmetric reduction *}
\thanks{* This work was supported by NSF REU Grant 2244488}

\author[D. Connolly]{Darien Connolly}
\address{Department of Mathematics,
SUNY Geneseo, Geneseo, NY 14454}
\email{dmc29@geneseo.edu}

\author[C. George]{Calvin George}
\address{Department of Mathematics, Dartmouth College, Hanover, NH 03755}
\email{calvin.d.george.24@dartmouth.edu}

\author[X. Hou]{Xiang-dong Hou}
\address{Department of Mathematics and Statistics,
University of South Florida, Tampa, FL 33620}
\email{xhou@usf.edu}

\author[A. Madro]{Adam Madro}
\address{Department of Mathematics and Statistics, Boston University, Boston, MA 02215} \email{amadro@bu.edu}

\author[V. Pallozzi Lavorante]{Vincenzo Pallozzi Lavorante}
\address{Department of Mathematics and Statistics,
University of South Florida, Tampa, FL 33620}
\email{vincenzop@usf.edu}



\begin{abstract}
An irreducible polynomial $f\in\f_q[X]$ of degree $n$ is {\em normal} over $\f_q$ if and only if its roots $r, r^q,\dots,r^{q^{n-1}}$ satisfy the condition $\Delta_n(r, r^q,\dots,r^{q^{n-1}})\ne 0$, where $\Delta_n(X_0,\dots,X_{n-1})$ is the $n\times n$ circulant determinant. By finding a suitable {\em symmetrization} of $\Delta_n$ (A multiple of $\Delta_n$ which is symmetric in $X_0,\dots,X_{n-1}$), we obtain a condition on the coefficients of $f$ that is sufficient for $f$ to be normal. This approach works well for $n\le 5$ but encounters computational difficulties when $n\ge 6$. In the present paper, we consider irreducible polynomials of the form $f=X^n+X^{n-1}+a\in\f_q[X]$. For $n=6$ and $7$, by an indirect method, we are able to find simple conditions on $a$ that are sufficient for $f$ to be normal. In a more general context, we also explore the normal polynomials of a finite Galois extension through the irreducible characters of the Galois group.
\end{abstract}

\maketitle

\section{Introduction}

Let $K$ be a finite Galois extension over $F$ with Galois group $G=\text{Aut}(K/F)$.  An element $\alpha\in K$ is said to be {\em normal} over $F$ (with respect to $K$) if $\{\sigma(\alpha):\sigma\in G\}$ forms a basis of $K$ over $F$; in this case, $\{\sigma(\alpha):\sigma\in G\}$ is called a {\em normal basis} of $K$ over $F$. Normal bases of finite Galois extension always exist \cite[\S VI.13]{Lang-2002}. The notion of normal bases appears in the context of Galois modules and leads to a subtle refinement in algebraic number theory \cite{Frohlich-1983, Snaith-1994}. If $\alpha\in K$ is normal over $F$ (with respect to $K$) and $f$ is its minimal polynomial over $F$, then $\alpha$ is separable over $F$, $F(\alpha)\subset K$, and $[F(\alpha):F]=\deg f=|G|=[K:F]$. It follows that $K=F(\alpha)$.  Therefore, one may speak of a separable element $\alpha$ over $F$ being normal over $F$ without mentioning $K$. The minimal polynomial of a normal element over $F$ is called a {\em normal polynomial} over $F$. Therefore, a normal polynomial over $F$ is a monic separable irreducible polynomial whose splitting field over $F$ is generated by one its roots and whose roots are linearly independent over $F$. When $F$ is a finite field $\f_q$, a normal polynomial over $\f_q$ is simply a monic irreducible polynomial whose roots are linearly independent over $\f_q$. Normal bases and normal polynomials over finite fields have applications in many areas and they constitute a well-studied topic in the theory of finite fields \cite{Gao-thesis-1993, Hou-DCC-2018, Lenstra-Schoof-MC-1987, Masuda-Moura-Panario-Thomson-IEEE-C-2008}, \cite[\S\S 5.2 -- 5.4]{Mullen-Panario-HF-2013}, \cite{Mullin-Onyszchuk-Vanstone-Wilson-DAM-1988, Pei-Wang-Omura-IEEE-IT-1986, Perlis-DMJ-1942}.

The number of normal polynomials of a given degree over $\f_q$ is known \cite[Theorem~3.73]{Lidl-Niederreiter-FF-1997}. For an irreducible polynomial $f=X^n+a_1X^{n-1}+\cdots+a_n\in\f_q[X]$ with roots $r_0=r,r_1=r^q,\dots,r_{n-1}=r^{q^{n-1}}$, $f$ is normal over $\f_q$ if and only if 
\[
\Delta_n(r_0,\dots,r_{n-1})\ne0,
\]
where 
\[
\Delta_n(X_0,\dots,X_{n-1})=\det\left[
\begin{matrix}
X_0&X_1&\cdots&X_{n-1}\cr
X_{n-1}&X_0&\cdots&X_{n-2}\cr
\vdots&\vdots&\ddots&\vdots\cr
X_1&X_2&\cdots&X_0
\end{matrix}\right]\in\Bbb Z[X_0\dots,X_{n-1}].
\]
However, it is more desirable to tell whether $f$ is normal from its coefficients $a_1,\dots,a_n$. This question was considered in \cite{Hou-preprint} with a naive approach which we briefly describe below.

The approach in \cite{Hou-preprint} is based on the notions of symmetrization and symmetric reduction. Given a polynomial $f(X_0,\dots,X_{n-1})\in\Bbb Z[X_0,\dots,X_{n-1}]$, a {\em symmetrization} of $f$ is a multiple of $f$ in $\Bbb Z[X_0,\dots,X_{n-1}]$ which is symmetric in $X_0,\dots X_{n-1}$. Given a symmetric polynomial $F(X_0,\dots,X_n)\in\Bbb Z[X_0,\dots,X_{n-1}]$, its {\em symmetric reduction} is the polynomial $H\in \Bbb Z[s_1,\dots,s_n]$ such that 
\[
F(X_0,\dots,X_{n-1})=H(s_1,\dots,s_n),
\]
where $s_i$ is the $i$th elementary symmetric polynomial in $X_0,\dots,X_{n-1}$. (Note:  The coefficient ring $\Bbb Z$ here can be replaced with any commutative ring.)

Let $\epsilon_n=e^{2\pi\sqrt{-1}/n}$. The polynomial $\Delta_n$ affords a factorization
\[
\Delta_n(X_0,\dots,X_{n-1})=\prod_{m\mid n}\Psi_m(Y_{n0},\dots,Y_{n,m-1}),
\]
where
\begin{equation}\label{Psi}
\Psi_n(X_0,\dots,X_{n-1})=\prod_{i\in(\Bbb Z/n\Bbb Z)^\times}\Bigl(\sum_{j\in\Bbb Z/n\Bbb Z}\epsilon_n^{ij}X_j\Bigr)\in\Bbb Z[X_0,\dots,X_{n-1}]
\end{equation}
and 
\[
Y_{ni}=\sum_{j=0}^{n/m-1}X_{i+mj},\quad 0\le i\le m-1.
\]
Let the symmetric group $S_n$ act on $\Bbb Z[X_0,\dots,X_{n-1}]$ by permuting the indeterminates $X_0,\dots,X_{n-1}$. The stabilizer of $\Psi_m(Y_{n0},\dots,Y_{n,m-1})\in\Bbb Z[X_0,\dots,X_{n-1}]$ in $S_n$, denoted by $\text{Stab}(\Psi_m(Y_{n0},\dots,Y_{n,m-1}))$, is determined. Consequently, 
\begin{equation}\label{Theta}
\Theta_{n,m}(X_0,\dots,X_{n-1}):=\prod_\phi\phi(\Psi_m(Y_{n0},\dots,Y_{n,m-1})),
\end{equation}
where $\phi$ runs over a left transversal of $\text{Stab}(\Psi_m(Y_{n0},\dots,Y_{n,m-1}))$ in $S_n$ (i.e., a system of representatives of the left cosets of $\text{Stab}(\Psi_m(Y_{n0},\dots,Y_{n,m-1}))$ in $S_n$), is a symmetrization of $\Psi_m(Y_{n0},\dots,Y_{n,m-1})$. The explicit formula for $\Theta_{n,m}$ is given in \cite{Hou-preprint}. Let $\theta_{n,m}\in\Bbb Z[s_1,\dots,s_n]$ be the symmetric reduction of $\Theta_{n,m}$, that is,
\begin{equation}\label{theta}
\Theta_{n,m}(X_0,\dots,X_{n-1})=\theta_{n,m}(s_1,\dots,s_n).
\end{equation}
If $p$ is prime, then $\theta_{n,p^tl}$ is a power of $\theta_{n,l}$ in characteristic $p$ (\cite[Remark~4.2]{Hou-preprint}). The main results of \cite{Hou-preprint} can be summarized as follows:

\begin{thm}\label{T1.1}\cite[Theorem~4.1]{Hou-preprint}
Let $f=X^n+a_1X^{n-1}+\cdots+a_n\in\f_q[X]$ be irreducible, where $p=\text{\rm char}\,\f_q$. If $\theta_{n,m}(a_1,\dots,a_n)\ne 0$ for all $m\mid n$ with $p\nmid m$, then $f$ is normal over $\f_q$.
\end{thm}

To recap, $\Theta_{n,m}$ are the symmetrizations of the factors of $\Delta_n$; $\theta_{n,m}$ is the symmetric reduction of $\Theta_{n,m}$. While $\Theta_{n,m}$ is explicitly given, $\theta_{n,m}$ is difficult to compute. The polynomials $\theta_{n,m}$ give rise to the sufficient condition that we seek for the normality of polynomials.

The usefulness of Theorem~\ref{T1.1} depends whether the polynomials $\theta_{n,m}$ (or their images in characteristic $p$) can be computed. In \cite{Hou-preprint}, $\theta_{n,m}$ were computed for all $m\mid n$ with $n\le 6$ except for $\theta_{6,6}$. Hence Theorem~\ref{T1.1} can be used to detect normality of irreducible polynomials of degree $\le 5$ and in cases $p=2$ or $3$, also of degree 6. However, for larger $n$ and $m$, the computation of the polynomial $\theta_{n,m}$, such as $\theta_{6,6}$ and $\theta_{7,7}$, becomes impractical as we will see in the next section.

If $f=X^n+a_1X^{n-1}+\cdots+a_n\in\f_q[X]$ is normal over $\f_q$, then $a_1\ne 0$. Since $f(a_1X)=a_1^n(X^n+X^{n-1}+\cdots+a_n')$, we may assume without loss of generality that $a_1=1$. In the present paper, we consider irreducible polynomials of the form 
\begin{equation}\label{f}
f=X^n+X^{n-1}+a\in\f_q[X]
\end{equation}
and we search for sufficient conditions on $a$ such that $f$ is normal over $\f_q$. From Theorem~\ref{T1.1}, we have
\begin{thm}\label{T1.2}
Assume that $f=X^n+X^{n-1}+a\in\f_q[X]$ is irreducible, where $p=\text{\rm char}\,\f_q$. If $\theta_{n,m}(1,0,\dots,0,a)\ne0$ for all $m\mid n$ with $m>1$ and $p\nmid m$, then $f$ is normal over $\f_q$.
\end{thm}

\noindent{\bf Note.} In Theorem~\ref{T1.2}, there is no need to consider $m=1$ since in general, $\theta_{n,1}(s_1,\dots,s_n)=s_1$, and hence $\theta_{n,1}(1,0,\dots,0,a)=1\ne0$.

\medskip
In Theorem~\ref{T1.2}, $\theta_{n,m}(1,0,\dots,0,a)$ is a polynomial in $a$ with coefficients in $\f_p$. The usefulness of Theorem~\ref{T1.2} of course depends on our ability to compute $\theta_{n,m}(1,0,\dots,0,a)$. Fortunately, there is an indirect method that allows us to compute $\theta_{n,m}(1,0,\dots,0,a)$ without having to compute $\theta_{n,m}(s_1,\dots,s_n)$ in general. We will describe this method in the next section. $\theta_{6,6}(1,0,\dots,0,a)$ and  $\theta_{7,7}(1,0,\dots,0,a)$ in $\f_p[a]$ for various primes $p$ are computed in Sections 3 and 4, respectively. These results give rise to practical criteria for an irreducible polynomial $f=X^n+X^{n-1}+a\in\f_q[X]$ ($n=6,7$) to be normal over $\f_q$; see Theorems~\ref{T3.1} and \ref{T4.1}.

In Section~5, we discuss normal polynomials of an arbitrary finite Galois extension with Galois group $G$. Basic facts about normal polynomials over finite fields extend naturally to the general situation. The determinant $\Delta_n$ played a central role for normal polynomials over finite fields. In the general case, $\Delta_n$ is replaced by the {\em group determinant} $\mathcal D_G$ of $G$. (In fact, $\Delta_n$ is the group determinant of the cyclic group $\Bbb Z/n\Bbb Z$.) The factors of $\mathcal D_G$ are given by the irreducible characters of $G$. When $G$ is small, the symmetrizations of the factors of $\mathcal D_G$ can be determined, and in some cases, the symmetric reductions of these symmetrizations can be computed. In such cases, we have criteria for the normality of polynomials in terms of their coefficients similar to the case of finite fields.

\section{An Indirect Method for Symmetric Reduction}

The main challenge that we face is the computation of the symmetric reduction $\theta_{n,m}$ (defined in \eqref{theta}) of the symmetric polynomial $\Theta_{n,m}$ (defined in \eqref{Theta}). We focus on $\theta_{6,6}$ and $\theta_{7,7}$. $\Theta_{6,6}$, given in \cite[Appendix~A3]{Hou-preprint}, is a product of 60 homogeneous polynomials of degree 2 in $X_0,\dots,X_5$:
\[
\Theta_{6,6}=\prod_{(i_1,i_2,i_3,i_4,i_5)}\Psi_6(X_0,X_{i_1},X_{i_2},X_{i_3},X_{i_4},X_{i_5}), 
\]
where
\[
\longequation{
\Psi_6=X_0^2+X_1 X_0-X_2 X_0-2 X_3 X_0-X_4 X_0+X_5
   X_0+X_1^2+X_2^2+X_3^2+X_4^2+X_5^2+X_1 X_2-X_1 X_3+X_2 X_3-2 X_1 X_4-X_2
   X_4+X_3 X_4-X_1 X_5-2 X_2 X_5-X_3 X_5+X_4 X_5,
}
\]  
and
\begin{align}\label{i1-i5}
&(i_1,i_2,i_3,i_4,i_5)=\\
&(1,3,4,5,2),(1,3,5,4,2),(1,4,3,5,2),(1,4,5,3,2),(1,5,3,4,2),(1,5,4,3,2),\cr
&(1,2,4,5,3),(1,2,5,4,3),(1,4,2,5,3),(1,4,5,2,3),(1,5,2,4,3),(1,5,4,2,3),\cr
&(1,2,3,5,4),(1,2,5,3,4),(1,3,2,5,4),(1,3,5,2,4),(1,5,2,3,4),(1,5,3,2,4),\cr
&(1,2,3,4,5),(1,2,4,3,5),(1,3,2,4,5),(1,3,4,2,5),(1,4,2,3,5),(1,4,3,2,5),\cr
&(2,1,4,5,3),(2,1,5,4,3),(2,4,1,5,3),(2,4,5,1,3),(2,5,1,4,3),(2,5,4,1,3),\cr
&(2,1,3,5,4),(2,1,5,3,4),(2,3,1,5,4),(2,3,5,1,4),(2,5,1,3,4),(2,5,3,1,4),\cr
&(2,1,3,4,5),(2,1,4,3,5),(2,3,1,4,5),(2,3,4,1,5),(2,4,1,3,5),(2,4,3,1,5),\cr
&(3,1,2,5,4),(3,1,5,2,4),(3,2,1,5,4),(3,2,5,1,4),(3,5,1,2,4),(3,5,2,1,4),\cr
&(3,1,2,4,5),(3,1,4,2,5),(3,2,1,4,5),(3,2,4,1,5),(3,4,1,2,5),(3,4,2,1,5),\cr
&(4,1,2,3,5),(4,1,3,2,5),(4,2,1,3,5),(4,2,3,1,5),(4,3,1,2,5),(4,3,2,1,5).\nonumber
\end{align}
$\Theta_{7,7}$ can be explicitly computed using equation (2.8) of \cite{Hou-preprint}; it is a product of 120 homogeneous polynomials of degree 6 in $X_0,\dots,X_6$; see Appendix~A1. The computation of the symmetric reductions $\theta_{6,6}(s_1,\dots,s_6)$ and $\theta_{7,7}(s_1,\dots,s_7)$ (over $\Bbb Z$ or in characteristic $p$) are impractical if not impossible. However, there is an indirect method that allows us to compute $\theta_{6,6}(1,0,\dots,0,s_6)$ and $\theta_{7,7}(1,0,\dots,0,s_7)$ in characteristic $p$. We describe this method in a more general setting.

Let $p$ be a prime and $\Theta(X_0,\dots,X_{n-1})\in\f_p[X_0,\dots,X_{n-1}]$ be a homogeneous symmetric polynomial of degree $d$. We have 
\begin{align*}
\Theta(X_0,\dots,X_{n-1})
\,&=\theta(s_1,\dots,s_n)\cr
&=\sum_{j\le\lfloor d/n\rfloor}b_js_1^{d-jn}s_n^j+\text{terms involving $s_2,\dots,s_{n-1}$},
\end{align*}
where $\theta$ is the symmetric reduction of $\Theta$, $s_i$ is the $i$th elementary symmetric polynomial in $X_0,\dots,X_{n-1}$, and $b_j\in\f_p$. Then
\begin{equation}\label{theta-bj}
\theta(1,0,\dots,0,s_n)=\sum_{j\le\lfloor d/n\rfloor}b_js_n^j
\end{equation}
and hence the question is to find $b_0,\dots,b_{\lfloor d/n\rfloor}\in\f_p$.

Given a polynomial $f\in\f_p[X]$ of relatively small degree, one can determine its roots explicitly in its splitting field. This is a unique feature not possessed by other fields. Our method exploits this distinctive advantage of finite fields. Take a ``sample'' polynomial $f=X^n-X^{n-1}+(-1)^nu\in\f_{p^k}[X]$ with $\f_p(u)=\f_{p^k}$, and let $r_0,\dots,r_{n-1}$ be the roots of $f$ (in some extension of $\f_{p^k}$). Then 
\begin{equation}\label{bj} 
\Theta(r_0,\dots,r_{n-1})=\sum_{j\le\lfloor d/n\rfloor}b_ju^j.
\end{equation}
Since $r_0,\dots,r_{n-1}$ are explicitly determined, we can compute $\Theta(r_0,\dots,r_{n-1})$. Since $\Theta(X_0,\dots,X_{n-1})$ is symmetric, $\Theta(r_0,\dots,r_{n-1})\in\f_{p^k}=\f_p(u)$. Hence
\begin{equation}\label{C(u)}
\Theta(r_0,\dots,r_{n-1})=[u^0,\dots,u^{k-1}]C(u)
\end{equation}
for some $C(u)\in\text{M}_{k\times 1}(\f_p)$. On the other hand, we can write
\begin{equation}\label{A(u)}
[u^0,\dots,u^{\lfloor d/n\rfloor}]=[u^0,\dots,u^{k-1}]A(u),
\end{equation}
where $A(u)\in\text{M}_{k\times(\lfloor d/n\rfloor+1)}(\f_p)$. By \eqref{C(u)} and \eqref{A(u)}, \eqref{bj} becomes
\[
[u^0,\dots,u^{k-1}]C(u)=[u^0,\dots,u^{k-1}]A(u)\left[\begin{matrix}b_0\cr\vdots\cr b_{\lfloor d/n\rfloor}\end{matrix}\right],
\]
i.e.,
\begin{equation}\label{LinEq}
A(u)\left[\begin{matrix}b_0\cr\vdots\cr b_{\lfloor d/n\rfloor}\end{matrix}\right]=C(u).
\end{equation}

Each sample polynomial $f=X^n-X^{n-1}+(-1)^nu$ gives rise to a linear system \eqref{LinEq} for $b_0,\dots,b_{\lfloor d/n\rfloor}$. Enough number of such linear systems will allow us to solve for $b_0,\dots,b_{\lfloor d/n\rfloor}$. More precisely, we have the following algorithm.

\begin{algo}\label{Alg2.1}\rm (An algorithm for computing $b_0,\dots,b_{\lfloor d/n\rfloor}$ in \eqref{theta-bj})
\begin{itemize}
\item[1.] Choose $u_1,\dots,u_m\in\overline\f_p$ (the algebraic closure of $\f_p$) with $[\f_p(u_i):\f_p]=k_i$ satisfying the following conditions:
\begin{itemize}
\item[(i)] The minimal polynomials of $u_1,\dots,u_m$ over $\f_p$, denoted by $g_1,\dots,g_m$, are all distinct.
\item[(ii)] $f_i:=X^n-X^{n-1}+(-1)^nu_i$ is irreducible over $\f_{p^{k_i}}$.
\item[(iii)] $\sum_{i=1}^mk_i\ge\lfloor d/n\rfloor+1$.
\end{itemize}
To satisfy (ii), use the following fact: If $\prod_{j=0}^{k_i-1}(X^n-X^{n-1}+(-1)^nu_i^{p^j})$ is irreducible over $\f_p$, then $X^n-X^{n-1}+(-1)^nu_i$ is irreducible over $\f_{p^{k_i}}$.

\medskip

\item[2.] For each $1\le i\le m$, compute the matrix $A(u_i)$ in \eqref{A(u)} by reducing $u^j$ ($0\le j\le\lfloor d/n\rfloor$) mod $g_i(u)$. Since $f_i$ is irreducible over $\f_{p^{k_i}}$, its roots are $r_j= X^{p^{k_ij}}$ modulo $f_i(X)$ and $g_i(u)$ ($0\le j\le n-1$). Compute the matrix $C(u_i)$  in \eqref{C(u)} by reducing $\Theta(r_0,\dots,r_{n-1})$ mod $f_i(X)$ and $g_i(u)$. 

\medskip

\item[3.] Solve the linear system
\begin{equation}\label{LinSys}
\left[
\begin{matrix} A(u_1)\cr\vdots\cr A(u_m)\end{matrix}\right]
\left[
\begin{matrix} b_0\cr\vdots\cr b_{\lfloor d/n\rfloor}\end{matrix}\right]=
\left[
\begin{matrix} C(u_1)\cr\vdots\cr C(u_m)\end{matrix}\right]
\end{equation}
for $b_0,\dots, b_{\lfloor d/n\rfloor}$. By Lemma~\ref{L2.2} below, 
\[
\text{rank}\left[
\begin{matrix} A(u_1)\cr\vdots\cr A(u_m)\end{matrix}\right]=\lfloor d/n\rfloor+1.
\]
Hence \eqref{LinSys} has a unique solution $b_0,\dots,b_{\lfloor d/n\rfloor}$.
\end{itemize}
\end{algo}

Algorithm~\ref{Alg2.1} is applied to $\Theta_{6,6}$ and $\Theta_{7,7}$ in the next two sections.

\begin{lem}\label{L2.2}
Let $N,m>0$ be integers and $u_1,\dots,u_m\in\overline\f_q$ (the algebraic closure of $\f_q$) be such that their minimal polynomials over $\f_q$ are all distinct. Let $[\f_q(u_i):\f_q]=k_i$ and write
\[
[u_i^0,u_i^1,\dots,u_i^{N-1}]=[u_i^0,\dots,u_i^{k_i-1}]A(u_i),
\]
where $A(u_i)\in\text{\rm M}_{k_i\times N}(\f_q)$. Then
\[
\text{\rm rank}\left[
\begin{matrix} A(u_1)\cr\vdots\cr A(u_m)\end{matrix}\right]=\min\{k_1+\cdots+k_m,N\}.
\]
\end{lem}

\begin{proof}
Let $g_i$ be the minimal polynomial of $u_i$ over $\f_q$. Let $a=(a_0,\dots,a_{N-1})\in\f_q^N$. We have
\begin{align*}
\left[\begin{matrix} A(u_1)\cr\vdots\cr A(u_m)\end{matrix}\right]a^T=0\ &\Leftrightarrow\ \left[\begin{matrix} u_1^0&\cdots&u_1^{N-1}\cr \vdots&&\vdots\cr u_m^0&\cdots&u_m^{N-1}\end{matrix}\right]a^T=0\cr
&\Leftrightarrow\ g_i \,\Bigm | \,\sum_{j=0}^{N-1}a_jx^j\ \text{for all}\ 1\le i\le m\cr
&\Leftrightarrow\ \prod_{i=1}^mg_i\,\Bigm | \,\sum_{j=0}^{N-1}a_jx^j.
\end{align*}
Since $\deg \prod_{i=1}^mg_i=\sum_{i=1}^mk_i$, we have
\[
\dim_{\f_q}\Bigl\{(a_0,\dots,a_{N-1})\in\f_q^N:\prod_{i=1}^mg_i\,\Bigm | \,\sum_{j=0}^{N-1}a_jx^j\Bigr\}=\max\Bigl\{N-\sum_{i=1}^mk_i,\,0\Bigr\}.
\]
Hence
\[
\text{rank}\left[
\begin{matrix} A(u_1)\cr\vdots\cr A(u_m)\end{matrix}\right]=N-\max\Bigl\{N-\sum_{i=1}^mk_i,\,0\Bigr\}=\min\Bigl\{\sum_{i=1}^mk_i,\,N\Bigr\}.
\]
\end{proof}

\medskip
\noindent{\bf Note.} In the above lemma, if $u_1$ and $u_2$ have the same minimal polynomial over $\f_q$, then $A(u_1)$ and $A(u_2)$ are row equivalent. The reason is that for $a=(a_0,\dots,a_{N-1})\in\f_q^N$, 
\[
A(u_1)a^T=0\ \Leftrightarrow\ [u_1^0,\dots,u_1^{N-1}]a^T=0\ \Leftrightarrow\ [u_2^0,\dots,u_2^{N-1}]a^T=0\ \Leftrightarrow\ A(u_2)a^T=0.
\]  

\medskip
\noindent{\bf Remark.} More generally, the method described in this section can be used to compute the symmetric reduction $\theta(s_1,\dots,s_n)$ where most $s_i$ are 0. For example, instead of \eqref{theta-bj}, we may consider
\[
\theta(1,0,\dots,0,s_{n-1},s_n)=\sum_{i(n-1)+jn\le d}b_{ij}s_{n-1}^is_n^j,\quad b_{ij}\in\f_p.
\]
The coefficients $b_{ij}$ can be determined using sample polynomials of the form $X^n-X^{n-1}+(-1)^{n-1}u+(-1)^nv$. Such results will produce criteria for irreducible polynomials of the form $X^n+X^{n-1}+aX+b\in\f_q[X]$ to be normal over $\f_q$ in terms of $a,b$.

\section{The Case $n=6$}

Let $f=X^6+X^5+a\in\f_q[X]$ be irreducible, where $\text{char}\,\f_q=p$. To apply Theorem~1.2, we need to compute $\theta_{6,m}(1,0,\dots,0,a)$ in characteristic $p$ for $m=2,3,6$. $\theta_{6,2}(s_1,\dots,s_6)$ and $\theta_{6,3}(s_1,\dots,s_6)$ have been determined in \cite[Appendix~A4]{Hou-preprint} as polynomials over $\Bbb Z$ regardless of $p$. In particular, we have 
\begin{equation}\label{theta62}
\theta_{6,2}(1,0,\dots,0,a)=(320 a-1)^2,
\end{equation}
\begin{equation}\label{theta63}
\theta_{6,3}(1,0,\dots,0,a)= 17222625 a^3+ 10935 a^2 + 1215 a+1.
\end{equation}
To compute $\theta_{6,6}(1,0,\dots,0,a)$, we apply Algorithm~\ref{Alg2.1} with $n=6$. Without much difficulty, we computed $\theta_{6,6}(1,0,\dots,0,a)$ in characteristic $p$ for $5\le p<1000$. ($p=2,3$ are unnecessary since they divide $6$.) For example, for $p=5$,
\begin{equation}\label{p=5}
\theta_{6,6}(1,0,\dots,0,a)=3(a^2+4a+2)^5.
\end{equation}
For $p=7$,
\begin{align}\label{p=7}
&\theta_{6,6}(1,0,\dots,0,a)=\\
&3 a^{14}+4 a^{13}+6 a^{12}+a^{11}+4 a^{10}+a^8+4 a^7+3 a^5+4 a^4+6 a^2+a+1.\nonumber
 \end{align}
For $p=11$,
\begin{align}\label{p=11}
&\theta_{6,6}(1,0,\dots,0,a)=\\
&4 (a^{19}+3 a^{18}+10 a^{17}+5 a^{16}+7 a^{15}+4 a^{14}+4 a^{13}+6 a^{12}+8
   a^{11}+10 a^{10}\cr
   &+4 a^9+10 a^8+7 a^7+5 a^6+5 a^5+8 a^4+2 a^3+a^2+2 a+3).\nonumber
\end{align}
The formulas of $\theta_{6,6}(1,0,\dots,0,a)$ for $5\le p\le 97$ are given in Appendix~A2. The following is Theorem~1.2 specified for $n=6$:

\begin{thm}\label{T3.1}
Assume that $f=X^6+X^5+a\in\f_q[X]$ is irreducible, where $\text{\rm char}\,\f_q=p\ge 5$. If $\prod_{m=2,3,6}\theta_{6,m}(1,0,\dots,0,a)\ne 0$, then $f$ is normal over $\f_q$.
\end{thm}

For example, when $p=5$, Theorem~\ref{T3.1} states that if $f=X^6+X^5+a\in\f_q[X]$ is irreducible and
\[
a^2+4a+2\ne0,
\]
then $f$ is normal over $\f_q$.

\medskip

When $p=7$, $\theta_{6,2}(1,0,\dots,0,a)=4(a+4)^2$ and $\theta_{6,3}(1,0,\dots,0,a)= a^2+4a+1$. If $f=X^6+X^5+a\in\f_q[X]$ is irreducible and
\[
\begin{cases}
a\ne 3,\cr
a^2+4a+1\ne0,\cr
3 a^{14}+4 a^{13}+6 a^{12}+a^{11}+4 a^{10}+a^8+4 a^7+3 a^5+4 a^4+6 a^2+a+1 \ne 0,
\end{cases}
\]
then $f$ is normal over $\f_q$. 

\medskip

When $p=11$, $\theta_{6,2}(1,0,\dots,0,a)=(a+10)^2$ and $\theta_{6,3}(1,0,\dots,0,a)=2 (a+7)(a^2+10a+4)$. If $f=X^6+X^5+a\in\f_q[X]$ is irreducible and
\[
\begin{cases}
a\ne 1,4,\cr
a^2+10a+4\ne0,\cr
4 (a^{19}+3 a^{18}+10 a^{17}+5 a^{16}+7 a^{15}+4 a^{14}+4 a^{13}+6 a^{12}+8
   a^{11}+10 a^{10}\cr
   +4 a^9+10 a^8+7 a^7+5 a^6+5 a^5+8 a^4+2 a^3+a^2+2 a+3)\ne 0,
\end{cases}
\]
then $f$ is normal over $\f_q$. 

\medskip

To conclude this section, we include the intermediate data in the computations of \eqref{p=5}, \eqref{p=7} and \eqref{p=11}. We follow the notation of Algorithm~\ref{Alg2.1}

\medskip

\noindent $\bullet$ $p=5$.

\medskip
Minimal polynomials of $u_i$:

\[
\renewcommand*{\arraystretch}{1.4}
\begin{tabular}{c|c|c}
\hline
$i$ &$k_i$ & $g_i(X)$\\ \hline
$1$& 2&$X^2+2$\\ 
$2$& 3&$X^3+X^2+2$\\
$3$& 3&$X^3+2 X^2+2 X+2$\\
$4$& 3&$X^3+2 X^2+4 X+2$\\
$5$& 2&$X^2+2 X+3$\\
$6$& 3&$X^3+3 X^2+2 X+3$\\
$7$& 3&$X^3+3 X+3$\\
$8$& 3&$X^3+4 X+3$\\
\hline
\end{tabular}
\]

\medskip
The matrices $[A(u_i)\ C(u_i)]$:

\medskip
$i=1$.
\[
\arraycolsep=1.6pt
\begin{array}{cccccccccccccccccccccc}
 1 & 0 & 3 & 0 & 4 & 0 & 2 & 0 & 1 & 0 & 3 & 0 & 4 & 0 & 2 & 0 & 1 & 0 & 3 & 0 & 4 &\quad 0\\
 0 & 1 & 0 & 3 & 0 & 4 & 0 & 2 & 0 & 1 & 0 & 3 & 0 & 4 & 0 & 2 & 0 & 1 & 0 & 3 & 0 &\quad 3\\
\end{array}
\]

\medskip
$i=2$.
\[
\arraycolsep=1.6pt
\begin{array}{cccccccccccccccccccccc}
 1 & 0 & 0 & 3 & 2 & 3 & 1 & 0 & 4 & 4 & 1 & 1 & 1 & 2 & 1 & 2 & 4 & 4 & 2 & 0 & 2 &\quad 0\\
 0 & 1 & 0 & 0 & 3 & 2 & 3 & 1 & 0 & 4 & 4 & 1 & 1 & 1 & 2 & 1 & 2 & 4 & 4 & 2 & 0 &\quad 1\\
 0 & 0 & 1 & 4 & 1 & 2 & 0 & 3 & 3 & 2 & 2 & 2 & 4 & 2 & 4 & 3 & 3 & 4 & 0 & 4 & 3 &\quad 0\\
\end{array}
\]

\medskip
$i=3$.
\[
\arraycolsep=1.6pt
\begin{array}{cccccccccccccccccccccc}
 1 & 0 & 0 & 3 & 4 & 1 & 4 & 2 & 1 & 1 & 2 & 2 & 0 & 2 & 2 & 2 & 3 & 1 & 3 & 1 & 0 &\quad 4\\
 0 & 1 & 0 & 3 & 2 & 0 & 0 & 1 & 3 & 2 & 3 & 4 & 2 & 2 & 4 & 4 & 0 & 4 & 4 & 4 & 1 &\quad 4\\
 0 & 0 & 1 & 3 & 2 & 3 & 4 & 2 & 2 & 4 & 4 & 0 & 4 & 4 & 4 & 1 & 2 & 1 & 2 & 0 & 4 &\quad 3\\
\end{array}
\]

\medskip
$i=4$.
\[
\arraycolsep=1.6pt
\begin{array}{cccccccccccccccccccccc}
 1 & 0 & 0 & 3 & 4 & 0 & 3 & 1 & 1 & 3 & 3 & 0 & 2 & 0 & 2 & 2 & 3 & 2 & 0 & 1 & 4 &\quad  0\\
 0 & 1 & 0 & 1 & 1 & 4 & 1 & 0 & 3 & 2 & 4 & 3 & 4 & 2 & 4 & 1 & 3 & 2 & 2 & 2 & 4 &\quad  0\\
 0 & 0 & 1 & 3 & 0 & 1 & 2 & 2 & 1 & 1 & 0 & 4 & 0 & 4 & 4 & 1 & 4 & 0 & 2 & 3 & 1 &\quad  2\\
\end{array}
\]

\medskip
$i=5$.
\[
\arraycolsep=1.6pt
\begin{array}{cccccccccccccccccccccc}
 1 & 0 & 2 & 1 & 2 & 3 & 3 & 0 & 1 & 3 & 1 & 4 & 4 & 0 & 3 & 4 & 3 & 2 & 2 & 0 & 4 &\quad 0\\
 0 & 1 & 3 & 1 & 4 & 4 & 0 & 3 & 4 & 3 & 2 & 2 & 0 & 4 & 2 & 4 & 1 & 1 & 0 & 2 & 1 &\quad 4\\
\end{array}
\]

\medskip
$i=6$.
\[
\arraycolsep=1.6pt
\begin{array}{cccccccccccccccccccccc}
 1 & 0 & 0 & 2 & 4 & 4 & 4 & 3 & 1 & 4 & 2 & 3 & 0 & 3 & 2 & 3 & 3 & 4 & 3 & 4 & 0 &\quad 0\\
 0 & 1 & 0 & 3 & 3 & 0 & 0 & 1 & 2 & 2 & 2 & 4 & 3 & 2 & 1 & 4 & 0 & 4 & 1 & 4 & 4 &\quad 1\\
 0 & 0 & 1 & 2 & 2 & 2 & 4 & 3 & 2 & 1 & 4 & 0 & 4 & 1 & 4 & 4 & 2 & 4 & 2 & 0 & 4 &\quad 1\\
\end{array}
\]

\medskip
$i=7$.
\[
\arraycolsep=1.6pt
\begin{array}{cccccccccccccccccccccc}
 1 & 0 & 0 & 2 & 0 & 4 & 4 & 3 & 1 & 4 & 3 & 0 & 4 & 1 & 3 & 0 & 3 & 1 & 1 & 3 & 4 &\quad 3\\
 0 & 1 & 0 & 2 & 2 & 4 & 3 & 2 & 4 & 0 & 2 & 3 & 4 & 0 & 4 & 3 & 3 & 4 & 2 & 4 & 2 &\quad 4\\
 0 & 0 & 1 & 0 & 2 & 2 & 4 & 3 & 2 & 4 & 0 & 2 & 3 & 4 & 0 & 4 & 3 & 3 & 4 & 2 & 4 &\quad 4\\
\end{array}
\]

\medskip
$i=8$.
\[
\arraycolsep=1.6pt
\begin{array}{cccccccccccccccccccccc}
 1 & 0 & 0 & 2 & 0 & 2 & 4 & 2 & 3 & 0 & 2 & 1 & 2 & 0 & 4 & 4 & 4 & 2 & 2 & 0 & 1 &\quad 1\\
 0 & 1 & 0 & 1 & 2 & 1 & 4 & 0 & 1 & 3 & 1 & 0 & 2 & 2 & 2 & 1 & 1 & 0 & 3 & 2 & 3 &\quad 0\\
 0 & 0 & 1 & 0 & 1 & 2 & 1 & 4 & 0 & 1 & 3 & 1 & 0 & 2 & 2 & 2 & 1 & 1 & 0 & 3 & 2 &\quad 3\\
\end{array}
\]

\medskip

\noindent $\bullet$ $p=7$.

\medskip
Minimal polynomials of $u_i$:

\[
\renewcommand*{\arraystretch}{1.4}
\begin{tabular}{c|c|c}
\hline
$i$ &$k_i$ & $g_i(X)$\\ \hline

$1$& $3$& $X^3+5 X^2+X+1$\\
$2$& $3$& $X^3+2 X+1$\\
$3$& $2$& $X^2+3 X+1$\\
$4$& $3$& $X^3+4 X+1$\\
$5$& $3$& $X^3+6 X^2+4 X+1$\\
$6$& $2$& $X^2+2 X+2$\\
$7$& $3$& $X^3+3 X^2+2 X+2$\\
$8$& $3$& $X^3+2 X^2+3$\\

\hline
\end{tabular}
\]

\medskip
The matrices $[A(u_i)\ C(u_i)]$:

\medskip
$i=1$.
\[
\arraycolsep=1.6pt
\begin{array}{ccccccccccccccccccccccc}
 1 & 0 & 0 & 6 & 5 & 4 & 4 & 6 & 4 & 5 & 0 & 5 & 5 & 5 & 0 & 4 & 3 & 2 & 4 & 3 & 0 & \quad 4 \\
 0 & 1 & 0 & 6 & 4 & 2 & 1 & 3 & 3 & 2 & 5 & 5 & 3 & 3 & 5 & 4 & 0 & 5 & 6 & 0 & 3 & \quad 3 \\
 0 & 0 & 1 & 2 & 3 & 3 & 1 & 3 & 2 & 0 & 2 & 2 & 2 & 0 & 3 & 4 & 5 & 3 & 4 & 0 & 0 & \quad 2 \\
\end{array}
\]

\medskip
$i=2$.
\[
\arraycolsep=1.6pt
\begin{array}{ccccccccccccccccccccccc}
 1 & 0 & 0 & 6 & 0 & 2 & 1 & 3 & 3 & 0 & 5 & 4 & 4 & 1 & 2 & 1 & 2 & 3 & 2 & 6 & 0 & \quad 3 \\
 0 & 1 & 0 & 5 & 6 & 4 & 4 & 0 & 2 & 3 & 3 & 6 & 5 & 6 & 5 & 4 & 5 & 1 & 0 & 0 & 6 & \quad 0 \\
 0 & 0 & 1 & 0 & 5 & 6 & 4 & 4 & 0 & 2 & 3 & 3 & 6 & 5 & 6 & 5 & 4 & 5 & 1 & 0 & 0 & \quad 2 \\
\end{array}
\]

\medskip
$i=3$.
\[
\arraycolsep=1.6pt
\begin{array}{ccccccccccccccccccccccc}
 1 & 0 & 6 & 3 & 6 & 0 & 1 & 4 & 1 & 0 & 6 & 3 & 6 & 0 & 1 & 4 & 1 & 0 & 6 & 3 & 6 & \quad 4 \\
 0 & 1 & 4 & 1 & 0 & 6 & 3 & 6 & 0 & 1 & 4 & 1 & 0 & 6 & 3 & 6 & 0 & 1 & 4 & 1 & 0 & \quad 5 \\
\end{array}
\]

\medskip
$i=4$.
\[
\arraycolsep=1.6pt
\begin{array}{ccccccccccccccccccccccc}
 1 & 0 & 0 & 6 & 0 & 4 & 1 & 5 & 6 & 0 & 6 & 1 & 4 & 4 & 4 & 1 & 1 & 6 & 2 & 3 & 0 & \quad 4 \\
 0 & 1 & 0 & 3 & 6 & 2 & 1 & 0 & 1 & 6 & 3 & 3 & 3 & 6 & 6 & 1 & 5 & 4 & 0 & 0 & 3 & \quad 2 \\
 0 & 0 & 1 & 0 & 3 & 6 & 2 & 1 & 0 & 1 & 6 & 3 & 3 & 3 & 6 & 6 & 1 & 5 & 4 & 0 & 0 & \quad 3 \\
\end{array}
\]

\medskip
$i=5$.
\[
\arraycolsep=1.6pt
\begin{array}{ccccccccccccccccccccccc}
 1 & 0 & 0 & 6 & 6 & 3 & 1 & 4 & 4 & 1 & 2 & 1 & 6 & 0 & 3 & 4 & 6 & 1 & 1 & 5 & 0 & \quad 3 \\
 0 & 1 & 0 & 3 & 2 & 4 & 0 & 3 & 6 & 1 & 2 & 6 & 4 & 6 & 5 & 5 & 0 & 3 & 5 & 0 & 5 & \quad 4 \\
 0 & 0 & 1 & 1 & 4 & 6 & 3 & 3 & 6 & 5 & 6 & 1 & 0 & 4 & 3 & 1 & 6 & 6 & 2 & 0 & 0 & \quad 3 \\
\end{array}
\]

\medskip
$i=6$.
\[
\arraycolsep=1.6pt
\begin{array}{ccccccccccccccccccccccc}
 1 & 0 & 5 & 4 & 3 & 0 & 1 & 5 & 2 & 0 & 3 & 1 & 6 & 0 & 2 & 3 & 4 & 0 & 6 & 2 & 5 & \quad 1 \\
 0 & 1 & 5 & 2 & 0 & 3 & 1 & 6 & 0 & 2 & 3 & 4 & 0 & 6 & 2 & 5 & 0 & 4 & 6 & 1 & 0 & \quad 5 \\
\end{array}
\]

\medskip
$i=7$.
\[
\arraycolsep=1.6pt
\begin{array}{ccccccccccccccccccccccc}
 1 & 0 & 0 & 5 & 6 & 0 & 6 & 5 & 1 & 3 & 0 & 6 & 4 & 4 & 3 & 3 & 5 & 1 & 2 & 3 & 6 & \quad 3 \\
 0 & 1 & 0 & 5 & 4 & 6 & 6 & 4 & 6 & 4 & 3 & 6 & 3 & 1 & 0 & 6 & 1 & 6 & 3 & 5 & 2 & \quad 6 \\
 0 & 0 & 1 & 4 & 0 & 4 & 1 & 3 & 2 & 0 & 4 & 5 & 5 & 2 & 2 & 1 & 3 & 6 & 2 & 4 & 0 & \quad 6 \\
\end{array}
\]

\medskip
$i=8$.
\[
\arraycolsep=1.6pt
\begin{array}{ccccccccccccccccccccccc}
 1 & 0 & 0 & 4 & 6 & 2 & 5 & 0 & 1 & 4 & 6 & 6 & 4 & 2 & 6 & 4 & 0 & 3 & 3 & 1 & 3 & \quad 0 \\
 0 & 1 & 0 & 0 & 4 & 6 & 2 & 5 & 0 & 1 & 4 & 6 & 6 & 4 & 2 & 6 & 4 & 0 & 3 & 3 & 1 & \quad 2 \\
 0 & 0 & 1 & 5 & 4 & 3 & 0 & 2 & 1 & 5 & 5 & 1 & 4 & 5 & 1 & 0 & 6 & 6 & 2 & 6 & 5 & \quad 3 \\
\end{array}
\]

\medskip

\noindent $\bullet$ $p=11$.

\medskip
Minimal polynomials of $u_i$:

\[
\renewcommand*{\arraystretch}{1.4}
\begin{tabular}{c|c|c}
\hline
$i$ &$k_i$ & $g_i(X)$\\ \hline

$1$& $3$& $X^3+5 X^2+1$\\
$2$& $3$& $X^3+7 X^2+1$\\
$3$& $3$& $X^3+3 X^2+5 X+1$\\
$4$& $3$& $X^3+6 X^2+5 X+1$\\
$5$& $3$& $X^3+4 X^2+6 X+1$\\
$6$& $3$& $X^3+8 X^2+6 X+1$\\
$7$& $3$& $X^3+8 X+1$\\

\hline
\end{tabular}
\]

\medskip
The matrices $[A(u_i)\ C(u_i)]$:

\medskip
$i=1$.
\[
\arraycolsep=2.8pt
\begin{array}{ccccccccccccccccccccccc}
 1 & 0 & 0 & 10 & 5 & 8 & 5 & 3 & 10 & 0 & 8 & 5 & 8 & 7 & 4 & 5 & 1 & 2 & 7 & 8 & 2 & \quad 0 \\
 0 & 1 & 0 & 0 & 10 & 5 & 8 & 5 & 3 & 10 & 0 & 8 & 5 & 8 & 7 & 4 & 5 & 1 & 2 & 7 & 8 & \quad 3 \\
 0 & 0 & 1 & 6 & 3 & 6 & 8 & 1 & 0 & 3 & 6 & 3 & 4 & 7 & 6 & 10 & 9 & 4 & 3 & 9 & 6 & \quad 10 \\
\end{array}
\]

\medskip
$i=2$.
\[
\arraycolsep=3pt
\begin{array}{ccccccccccccccccccccccc}
 1 & 0 & 0 & 10 & 7 & 6 & 3 & 5 & 3 & 9 & 9 & 0 & 2 & 10 & 7 & 4 & 6 & 6 & 9 & 8 & 4 & \quad 3 \\
 0 & 1 & 0 & 0 & 10 & 7 & 6 & 3 & 5 & 3 & 9 & 9 & 0 & 2 & 10 & 7 & 4 & 6 & 6 & 9 & 8 & \quad 8 \\
 0 & 0 & 1 & 4 & 5 & 8 & 6 & 8 & 2 & 2 & 0 & 9 & 1 & 4 & 7 & 5 & 5 & 2 & 3 & 7 & 4 & \quad 1 \\
\end{array}
\]

\medskip
$i=3$.
\[
\arraycolsep=2.75pt
\begin{array}{ccccccccccccccccccccccc}
 1 & 0 & 0 & 10 & 3 & 7 & 9 & 1 & 0 & 8 & 8 & 2 & 1 & 1 & 1 & 2 & 10 & 3 & 5 & 4 & 4 & \quad 5 \\
 0 & 1 & 0 & 6 & 3 & 5 & 8 & 3 & 1 & 7 & 4 & 7 & 7 & 6 & 6 & 0 & 8 & 3 & 6 & 3 & 2 & \quad 7 \\
 0 & 0 & 1 & 8 & 4 & 2 & 10 & 0 & 3 & 3 & 9 & 10 & 10 & 10 & 9 & 1 & 8 & 6 & 7 & 7 & 4 & \quad 3 \\
\end{array}
\]

\medskip
$i=4$.
\[
\arraycolsep=2.4pt
\begin{array}{ccccccccccccccccccccccc}
 1 & 0 & 0 & 10 & 6 & 2 & 3 & 10 & 0 & 2 & 0 & 1 & 3 & 10 & 1 & 7 & 9 & 9 & 4 & 10 & 10 & \quad 7 \\
 0 & 1 & 0 & 6 & 7 & 5 & 6 & 9 & 10 & 10 & 2 & 5 & 5 & 9 & 4 & 3 & 8 & 10 & 7 & 10 & 5 & \quad 8 \\
 0 & 0 & 1 & 5 & 9 & 8 & 1 & 0 & 9 & 0 & 10 & 8 & 1 & 10 & 4 & 2 & 2 & 7 & 1 & 1 & 4 & \quad 1 \\
\end{array}
\]

\medskip
$i=5$.
\[
\arraycolsep=2pt
\begin{array}{ccccccccccccccccccccccc}
 1 & 0 & 0 & 10 & 4 & 1 & 6 & 10 & 0 & 0 & 1 & 7 & 10 & 5 & 1 & 0 & 0 & 10 & 4 & 1 & 6 & \quad 9 \\
 0 & 1 & 0 & 5 & 1 & 10 & 4 & 0 & 10 & 0 & 6 & 10 & 1 & 7 & 0 & 1 & 0 & 5 & 1 & 10 & 4 & \quad 9 \\
 0 & 0 & 1 & 7 & 10 & 5 & 1 & 0 & 0 & 10 & 4 & 1 & 6 & 10 & 0 & 0 & 1 & 7 & 10 & 5 & 1 & \quad 1 \\
\end{array}
\]

\medskip
$i=6$.
\[
\arraycolsep=3pt
\begin{array}{ccccccccccccccccccccccc}
 1 & 0 & 0 & 10 & 8 & 8 & 10 & 7 & 8 & 5 & 4 & 7 & 3 & 7 & 7 & 9 & 0 & 5 & 6 & 10 & 0 & \quad 7 \\
 0 & 1 & 0 & 5 & 3 & 1 & 2 & 8 & 0 & 5 & 7 & 2 & 3 & 1 & 5 & 6 & 9 & 8 & 8 & 0 & 10 & \quad 2 \\
 0 & 0 & 1 & 3 & 3 & 1 & 4 & 3 & 6 & 7 & 4 & 8 & 4 & 4 & 2 & 0 & 6 & 5 & 1 & 0 & 0 & \quad 5 \\
\end{array}
\]

\medskip
$i=7$.
\[
\arraycolsep=2.4pt
\begin{array}{ccccccccccccccccccccccc}
 1 & 0 & 0 & 10 & 0 & 8 & 1 & 2 & 6 & 5 & 5 & 9 & 10 & 0 & 10 & 1 & 8 & 4 & 1 & 4 & 10 & \quad 9 \\
 0 & 1 & 0 & 3 & 10 & 9 & 5 & 6 & 6 & 2 & 1 & 0 & 1 & 10 & 3 & 7 & 10 & 7 & 1 & 0 & 7 & \quad 8 \\
 0 & 0 & 1 & 0 & 3 & 10 & 9 & 5 & 6 & 6 & 2 & 1 & 0 & 1 & 10 & 3 & 7 & 10 & 7 & 1 & 0 & \quad 9 \\
\end{array}
\]

\section{The Case $n=7$}

For $n=7$, Theorem~\ref{T1.2} becomes

\begin{thm}\label{T4.1}
Assume that $f=X^7+X^6+a\in\f_q[X]$ is irreducible, where $\text{\rm char}\,\f_q=p\ne 7$. If $\theta_{7,7}(1,0,\dots,0,a)\ne0$, then $f$ is normal over $\f_q$.
\end{thm}

Using Algorithm~\ref{Alg2.1}, we have computed $\theta_{7,7}(1,0,\dots,0,a)$ for $p\le 47$ ($p\ne 7$). (The algorithm can be implemented for larger $p$ with only incremental complexity.) When $p=2$,
\begin{align*}
&\theta_{7,7}(1,0,\dots,0,a)=\cr
&(a^4+a^3+a^2+a+1)^2 (a^5+a^4+a^3+a+1)^2
     (a^8+a^7+a^2+a+1)^2\cr
& (a^8+a^7+a^4+a^3+a^2+a+1)^2
     (a^{10}+a^8+a^7+a^6+a^5+a^2+1)^2.
\end{align*}
When $p=3$,
\begin{align*}
\theta_{7,7}(1,0,\dots,0,a)=\,
&2 (a+2)^{18} (a^9+2 a^6+2 a^5+2 a^4+a^3+a+2)^3\cr
& (a^{10}+2a^9+a^8+a^6+2 a^4+2 a^3+1)^3.
\end{align*}
When $p=5$,
\begin{align*}
&\theta_{7,7}(1,0,\dots,0,a)=\cr
&(a+1)^3 (a^3+4 a^2+3 a+4) (a^{89}+4 a^{88}+4 a^{87}+3 a^{86}+2 a^{85}+3
   a^{83}+3 a^{82}+4 a^{81}\cr
   &+4 a^{80}+2 a^{79}+4 a^{78}+3 a^{77}+2 a^{76}+a^{75}+3
   a^{74}+2 a^{73}+4 a^{72}+a^{70}+3 a^{69}+4 a^{68}\cr
   &+2 a^{67}+a^{66}+2 a^{65}+2 a^{64}+4
   a^{62}+2 a^{61}+4 a^{60}+4 a^{59}+a^{58}+a^{56}+a^{55}+3 a^{53}\cr
   &+3 a^{51}+4 a^{50}+2
   a^{49}+2 a^{47}+2 a^{46}+3 a^{45}+2 a^{44}+a^{43}+4 a^{42}+a^{41}+4 a^{40}+2 a^{38}\cr
   &+4a^{37}+4 a^{36}+2 a^{35}+2 a^{34}+3 a^{33}+3 a^{32}+2 a^{31}+3
   a^{30}+a^{29}+a^{28}+a^{27}+2 a^{26}\cr
   &+2 a^{25}+a^{24}+4 a^{23}+3 a^{22}+4
   a^{21}+a^{20}+2 a^{16}+3 a^{15}+2 a^{14}+3 a^{13}+2 a^{12}+a^{11}\cr
   &+a^{10}+3 a^9+4 a^6+2
   a^5+2 a^4+2 a^3+a+4).
\end{align*}
The intermediate data in the computation of $\theta_{7,7}(1,0,\dots,0,a)$ with $p=2$ are included in Appendix A4.

Theorems~\ref{T3.1} and \ref{T4.1} are simple criteria for the normality of irreducible polynomials of the form $X^n+X^{n-1}+a\in\f_q[X]$ ($n=6,7$). Are the sufficient conditions in Theorems~\ref{T3.1} and \ref{T4.1} close to being necessary? Let
\begin{align}\label{Irr}
\text{Irr}\,(q,n)=\,&\text{the number of irreducible polynomials}\\
&\text{of the form $X^n+X^{n-1}+a\in\f_q[X]$},\nonumber
\end{align}
\begin{align}\label{Norm}
\text{Norm}\,(q,n)=\,&\text{the number of normal polynomials}\\
&\text{of the form $X^n+X^{n-1}+a\in\f_q[X]$},\nonumber
\end{align}
and for $n=6,7$, let 
\begin{align}\label{Suff}
\text{Suff}\,(q,n)=\,&\text{the number of irreducible polynomials of the form}\\ &\text{$X^n+X^{n-1}+a\in\f_q[X]$ satisfying the conditions}\cr
&\text{in Theorem~\ref{T3.1} or \ref{T4.1}}.\nonumber
\end{align}
We computed these three numbers for $n=6,7$, $q\le 49$ and $\text{gcd}(q,n)=1$; see Appendix~A3. In the cases that we computed, $\text{Irr}\,(q,n)=\text{Norm}\,(q,n)$ most of the time and $\text{Norm}\,(q,n)=\text{Suff}\,(q,n)$ all the time. However, since the three numbers in these cases are all quite small, it is difficult to tell whether they are indicative of a general pattern.

\section{Normal Polynomials of a Finite Galois Extension}

Let $G=\{g_1,\dots,g_n\}$ be a finite group. To each $g\in G$ assign an indeterminate $X_g$. The {\em group determinant} of $G$ is
\[
\mathcal D_G(X_{g_1},\dots,X_{g_n})=\det(X_{g_ig_j^{-1}})\in\Bbb Z[X_{g_1},\dots,X_{g_n}],
\]
where $(X_{g_ig_j^{-1}})$ is the $n\times n$ matrix whose $(i,j)$ entry is $X_{g_ig_j^{-1}}$. Note that when $G=\Bbb Z/n\Bbb Z$, 
\[
\mathcal D_{\Bbb Z/n\Bbb Z}(X_0,\dots,X_{n-1})=\Delta_n(X_0,\dots,X_{n-1}).
\]
Let $F$ be any field and $f$ be a monic, separable irreducible polynomial over $F$ such that $F(\alpha)=S_F(f)$, where $\alpha$ is any root of $f$ and $S_F(f)$ demotes the splitting field of $f$ over $F$. Let $G=\text{Aut}(S_F(f)/F)=\{g_1,\dots,g_n\}$ be the Galois group of $f$ over $F$. It is well known that $f$ is normal over $F$ if and only if
\begin{equation}\label{Dg1}
\mathcal D_G(g_1(\alpha)\,\dots,g_n(\alpha))\ne 0.
\end{equation}
(See \cite[Lemma~3.2]{Hou-preprint}.)

Let $\rho_1,\dots,\rho_k$ be the irreducible representations of $G$ with $\deg\rho_i=m_i$. It is a fundamental fact from the representation theory that 
\begin{equation}\label{Dg2}
\mathcal D_G=\prod_{i=1}^k\Bigl(\det\Bigl(\sum_{g\in G}X_g\rho_i(g)\Big)\Bigr)^{m_i}=\prod_{i=1}^k(\det\mathcal X_i)^{m_i},
\end{equation}
where
\begin{equation}\label{Xi}
\mathcal X_i=\sum_{g\in G}X_g\rho_i(g)\in\text{M}_{m_i\times m_i}(\Bbb C).
\end{equation}
In fact, \eqref{Dg2} is the complete factorization of $\mathcal D_G$ in $\Bbb C[X_{g_1},\dots,X_{g_n}]$. For references on group determinants, see \cite{Formanek-Sibley-PAMS-1991, Frobenius-1896a, Frobenius-1896b, Hawkins-1971, Hawkins-1974, Johnson-MPCPS-1991}. The elementary symmetric polynomials in $X_1,\dots,X_m$ can be expressed as a polynomial in the power sums $p_j=X_1^j+\cdots+X_m^j$, $1\le j\le m$, with coefficients in $\Bbb Q$; see \cite[(2.14)$'$]{Macdonald-1995}. In particular, 
\begin{equation}\label{power-sum}
X_1\cdots X_m=(-1)^m\sum_{1t_1+2t_2+\cdots+mt_m=m}\;\prod_{j=1}^m\frac{(-p_j)^{t_j}}{t_j!\,j^{t_j}}.
\end{equation}
If $A$ be an $m\times m$ matrix over $\Bbb C$ with eigenvalues $\lambda_1,\dots,\lambda_m$, then
\[
p_j(\lambda_1,\dots,\lambda_m)=\lambda_1^j+\cdots+\lambda_m^j=\text{Tr}(A^j),
\]
and hence by \eqref{power-sum},
\begin{equation}\label{detA}
\det A=\lambda_1\cdots\lambda_m=(-1)^m\sum_{1t_1+2t_2+\cdots+mt_m=m}\;\prod_{j=1}^m\frac{(-\text{Tr}(A^j))^{t_j}}{t_j!\,j^{t_j}}.
\end{equation}
Therefore, 
\begin{equation}\label{detX}
\det \mathcal X_i=(-1)^{m_i}\sum_{1t_1+2t_2+\cdots+m_it_{m_i}=m_i}\;\prod_{j=1}^{m_i}\frac{(-\text{Tr}(\mathcal X_i^j))^{t_j}}{t_j!\,j^{t_j}}.
\end{equation}
For $1\le j\le m_i$,
\[
\mathcal X_i^j=\sum_{g\in G}\Bigl(\sum_{\substack{h_1,\dots,h_j\in G\cr h_1\dots h_j=g}}X_{h_1}\cdots X_{h_j}\Bigr)\rho_i(g).
\]
Let $\chi_i$ be the character of $\rho_i$. Then 
\begin{equation}\label{Tr}
\text{Tr}(\mathcal X_i^j)=\sum_{g\in G}\Bigl(\sum_{h_1\dots h_j=g}X_{h_1}\cdots X_{h_j}\Bigr)\chi_i(g).
\end{equation}
To summarize, $\mathcal D_G$ is given by \eqref{Dg2} in which $\det \mathcal X_i$ is given by \eqref{detX} and \eqref{Tr}. 

Let 
\begin{equation}\label{Fi}
F_i(X_{g_1},\dots,X_{g_n})=\det\mathcal X_i\in\Bbb C[X_{g_1},\dots,X_{g_n}].
\end{equation}
By \eqref{detX} and \eqref{Tr}, the coefficients of $F_i$ are algebraic numbers, by \eqref{Xi}, $F_i$ is monic in $X_e$, where $e$ is the identity of $G$, and by \eqref{Dg2}, $\prod_{i=1}^kF_i^{m_i}=\mathcal D_G\in\Bbb Z[X_{g_1},\dots,X_{g_n}]$. From these facts it follows that the coefficients of $F_i$ are algebraic integers.  Let $\frak o$ denote the ring of algebraic integers in $\Bbb C$ and let $E_i\in\frak o[X_{g_1},\dots,X_{g_n}]$ be a symmetrization of $F_i$. Let $H_i\in\frak o[s_1,\dots,s_n]$ be the symmetric reduction of $E_i$, i.e.,
\begin{equation}\label{Ei}
E_i(X_{g_1},\dots,X_{g_n})=H_i(s_1,\dots,s_n),
\end{equation}
where $s_i$ is the $i$th elementary symmetric polynomial in $X_{g_1},\dots,X_{g_n}$.
Then we have the following generalization of Theorem~\ref{T1.1}:

\begin{thm}\label{T5.1}
Let $F$ be a field and $f=X^n+a_1X^{n-1}+\cdots+a_n$ be a separable irreducible polynomial over $F$ such that $F(\alpha)=S_F(f)$, where $\alpha$ is any root of $f$. Let $G=\text{\rm Aut}(S_F(f)/F)$. If $H_i(a_1,\dots,a_n)\ne0$ for all $1\le i\le k$, where $H_i$ is defined in \eqref{Ei}, then $f$ is normal over $F$.
\end{thm}

\noindent{\bf Note.} Since $H_i(s_1,\dots,s_n)\in\frak o[s_1,\dots,s_n]$, $H_i(a_1,\dots,a_n)$ is meaningful regardless of the characteristic of $F$.

\medskip
In Theorem~\ref{T5.1}, the computation of $H_i$ totally depends on the Galois group $G$. Hence, when $G$ is cyclic, the question is no different from the case of finite fields; see \cite[Theorem~3.3]{Hou-preprint}. We will further discuss Theorem~\ref{T5.1} when $G$ is a noncyclic group of order $\le 8$. For this purpose, we need a technical lemma.

Let $m\mid n$ and partition $\Bbb Z/n\Bbb Z$ into blocks $I_i=\{i+mj:0\le j\le n/m-1\}$, $0\le i\le m-1$. Let $S_n$ be the permutation group of $\Bbb Z/n\Bbb Z$. The {\em wreath product} $S_m \kern1pt \wr S_{n/m}$ is the subgroup of $S_n$ consisting of all permutations that are obtained as follows: first permute the blocks $I_1,\dots, I_m$ using a permutation from $S_m$; then permute the elements in each block $I_i$ independently. More precisely,
\[
S_m\wr S_{n/m}=\{(\sigma;\sigma_0,\dots,\sigma_{m-1}):\sigma\in S_m, \sigma_i\in S_{n/m}, 0\le i\le m-1\},
\]
where $(\sigma;\sigma_0,\dots,\sigma_{m-1})$ maps $i+mj$ to $\sigma(i)+m\sigma_i(j)$. A left transversal of $S_m\wr S_{n/m}$ in $S_n$ has been determined in \cite[\S~2]{Hou-preprint}. Let $\mathcal P_{n,m}$ be the set of all unordered partitions of $\{0,1,\dots,n-1\}$ into $m$ parts of size $n/m$. For each $\{P_0,\dots,P_{m-1}\}\in\mathcal P_{n,m}$, choose a permutation $\phi_{\{P_0,\dots,P_{m-1}\}}\in S_n$ which maps $I_i$ to $P_i$, $0\le i\le m-1$. Then
\begin{equation}\label{Enm}
\mathcal E_{n,m}:=\{\phi_{\{P_0,\dots,P_{m-1}\}}:\{P_0,\dots,P_{m-1}\}\in\mathcal P_{n,m}\}
\end{equation}
is a left transversal of $S_m\wr S_{n/m}$ in $S_n$. 

\begin{lem}\label{L5.2}
Let $m\mid n$ with $n\ge 3$. Let $A(Y_0,\dots,Y_{m-1})\in\Bbb C[Y_0,\dots,Y_{m-1}]$ be such that for each $\sigma\in S_m$,
\begin{equation}\label{+-A}
A(Y_{\sigma(0)},\dots,Y_{\sigma(m-1)})=\pm A(Y_0,\dots,Y_{m-1}).
\end{equation}
Let $B(Z_0,\dots,Z_{n/m-1})\in\Bbb C[Z_0,\dots,Z_{n/m-1}]$ be symmetric. Let 
\[
B_i=B(X_i,X_{i+m},\dots,X_{i+(n/m-1)m}),\quad 0\le i\le m-1,
\]
and
\[
F=A(B_0,\dots,B_{m-1})\in\Bbb C[X_0,\dots,X_{n-1}].
\]
Further assume that 
\begin{equation}\label{FX0}
\deg_{X_0}F(X_0,0,\dots,0)>0.
\end{equation}
Then
\[
\prod_{\phi\in\mathcal E_{n,m}}\phi(F)\in\Bbb C[X_0,\dots,X_{n-1}]
\]
is symmetric, where $\mathcal E_{n,m}$ is the left transversal of $S_m\wr S_{n/m}$ in $S_n$ given in \eqref{Enm}. 
\end{lem}

\begin{proof}
For each $\sigma\in S_n$, $\{\sigma\phi:\phi\in\mathcal E_{n,m}\}$ is also a left transversal of $S_m\wr S_{n.m}$ in $S_n$, whence $\{\sigma\phi:\phi\in\mathcal E_{n,m}\}=\{\phi\sigma_\phi:\phi\in\mathcal E_{n,m}\}$, where each $\sigma_\phi\in S_m\wr S_{n,m}$. Therefore, 
\[
\sigma\Bigl(\prod_{\phi\in\mathcal E_{n,m}}\phi(F)\Bigr)=\prod_{\phi\in\mathcal E_{n,m}}\phi((\sigma_\phi(F)).
\]
By assumption, $\sigma_\phi(F)=\pm F$. Thus
\begin{equation}\label{sign}
\sigma\Bigl(\prod_{\phi\in\mathcal E_{n,m}}\phi(F)\Bigr)=\pm\prod_{\phi\in\mathcal E_{n,m}}\phi(F).
\end{equation}
Since $S_n$ is generated by transpositions and all transpositions are conjugate in $S_n$, to prove that $\prod_{\phi\in\mathcal E_{n,m}}\phi(F)$ is symmetric, it suffices to show that \eqref{sign} holds with a $+$ sign for {\em some} transposition $\sigma$. By \eqref{FX0}, 
\[
F(X_0,\dots,X_{n-1})=a_0X_0^t+\cdots+a_{n-1}X_{n-1}^t+\text{other terms},
\]
where $t>0$, $a_0\cdots a_{n-1}\ne0$ and the ``other terms'' are those of the form $x_i^s$, $s<t$, and those involving at least two variables. It follows that
\[
\prod_{\phi\in\mathcal E_{n,m}}\phi(F)=a_0'X_0^{t'}+\cdots+a_{n-1}'X_{n-1}^{t'}+\text{other terms},
\]
where $t'>0$ and $a_0'\dots a_{n-1}'\ne0$. Assume that
\[
\sigma\Bigl(\prod_{\phi\in\mathcal E_{n,m}}\phi(F)\Bigr)=-\prod_{\phi\in\mathcal E_{n,m}}\phi(F)
\]
for $\sigma=(0,1)$ and $(0,2)$. Then $a_1'=-a_0'=a_2'$. Thus, when $\sigma=(1,2)$, we must have 
\[
\sigma\Bigl(\prod_{\phi\in\mathcal E_{n,m}}\phi(F)\Bigr)=\prod_{\phi\in\mathcal E_{n,m}}\phi(F).
\]
\end{proof}

Now we take a closer look of Theorem~\ref{T5.1} when $G$ is a noncyclic group of order $\le 8$. The character tables of these groups are taken from \cite{James-Liebeck-2001}. In each case, we compute $F_i$ in \eqref{Fi} and its symmetrization $E_i$. Whenever possible, we also compute the symmetric reduction $H_i$ of $E_i$.

\medskip
\noindent{\bf Case 1.} $G=(\Bbb Z/2\Bbb Z)^2$. 

Table~\ref{Tb1} is the character table of $(\Bbb Z/2\Bbb Z)^2$. In general, in the character table of a finite group $G$, the entries of the first row are the representatives of the conjugacy classes of $G$, the entries of the second row are the sizes of the conjugacy classes. $\chi_1,\dots,\chi_k$ are the irreducible characters, $\chi_1$ is the principal character, and $\deg\chi_i=\chi_i(e)$, where $e$ is the identity of $G$. We always have $F_1=\sum_{g\in G}X_g$, $E_1=F_1$ and $H_1=s_1$. Hence we will ignore $F_1$.
\begin{table}[h]
\caption{Character table of $(\Bbb Z/2\Bbb Z)^2$}\label{Tb1}
\vspace{-2em}
\[
\renewcommand*{\arraystretch}{1.2}
\begin{tabular}{c|cccc}
\hline
conj. cls. rep. & (0,0) & (0,1) & (1,0) & (1,1) \\ 
conj. cls. size & 1 & 1& 1& 1\\ \hline
$\chi_1$ & $1$ & $1$ & $1$ & $1$ \\
$\chi_2$ & $1$ & $-1$ & $1$ & $-1$ \\
$\chi_3$ & $1$ & $1$ & $-1$ & $-1$ \\
$\chi_4$ & $1$ & $-1$ & $-1$ & $1$ \\ \hline
\end{tabular}
\]
\end{table}
We have
\begin{align*}
F_2\,&=\sum_{(i,j)\in G}(-1)^jX_{(i,j)},\cr
F_3\,&=\sum_{(i,j)\in G}(-1)^iX_{(i,j)},\cr
F_4\,&=\sum_{(i,j)\in G}(-1)^{i+j}X_{(i,j)}.
\end{align*}
Since $F_2,F_3,F_4$ are in the same $S_4$-orbit, they have the same symmetrizations. We relabel the variables $X_g$ by renaming the group elements $g$:
\[
((0,0),(1,0),(0,1),(1,1))\to (0,2,1,3).
\]
Then $F_2=X_0+X_2-(X_1+X_3)$. Using Lemma~\ref{L5.2} (with $A=Y_0-Y_1$, $B=Z_0+Z_1$), 
we may choose
\[
E_2=\prod_{\phi\in\mathcal E_{4,2}}\phi(F_2)=(X_0+X_1-X_2-X_3)(X_0+X_2-X_1-X_3)(X_0+X_3-X_1-X_2).
\]
The symmetric reduction of $E_2$ is
\[
H_2=s_1^2-4s_1s_2+8s_3.
\]

\medskip
\noindent{\bf Case 2.} $G=S_3$. 

\begin{table}[h]
\caption{Character table of $S_3$}\label{Tb2}
\vspace{-2em}
\[
\renewcommand*{\arraystretch}{1.2}
\begin{tabular}{c|ccc}
\hline
conj. cls. rep. & (1) & (12) & (123)  \\ 
conj. cls. size & 1 & 3& 2\\ \hline
$\chi_1$ & $1$ & $1$ & $1$  \\
$\chi_2$ & $1$ & $-1$ & $1$  \\
$\chi_3$ & $2$ & $0$ & $-1$ \\ \hline
\end{tabular}
\]
\end{table}
Let 
\begin{align*}
G_1\,&=\{(1),(123),(132)\}=A_3,\cr
G_2\,&=\{(12),(23),(13)\}=S_3\setminus A_3.
\end{align*}
We have 
\[
F_2=\sum_{g\in G_1}X_g-\sum_{g\in G_2}X_g.
\]
To compute $F_3$, we first find $\text{Tr}(\mathcal X_3^1)$ and $\text{Tr}(\mathcal X_3^2)$:
\[
\text{Tr}(\mathcal X_3^1)=2X_{(1)}-X_{(123)}-X_{(132)},
\]
\begin{align*}
&\text{Tr}(\mathcal X_3^2)=\sum_{g\in G}\Bigl(\sum_{h_1h_2=g}X_{h_1}X_{h_2}\Bigr)\chi_3(g)\cr
&=2(X_{(1)}^2+X_{(12)}^2+X_{(23)}^2+X_{(13)}^2)-X_{(123)}^2-X_{(132)}^2-2(X_{(1)}X_{(123)}+X_{(1)}X_{(132)})\cr
&\kern1.1em -2(X_{(12)}X_{(13)}+X_{(12)}X_{(23)}+X_{(13)}X_{(23)})+4X_{(123)}X_{(132)}.
\end{align*}
Then
\begin{align*}
F_3\,&=\frac{(-\text{Tr}(\mathcal X_3^1))^2}{2!1^2}+\frac{(-\text{Tr}(\mathcal X_3^2))^1}{1!2^1}\cr
&=\sum_{g\in G_1}X_g^2-\sum_{g\in G_2}X_g^2-\sum_{\substack{\{h_1,h_2\}\subset G_1\cr h_1\ne h_2}}X_{h_1}X_{h_2}+\sum_{\substack{\{h_1,h_2\}\subset G_2\cr h_1\ne h_2}}X_{h_1}X_{h_2}.
\end{align*}
Rename the elements of $G$ such that $G_1=\{0,2,4\}$ and $G_2=\{1,3,5\}$. Then
\[
F_2=X_0+X_2+X_4-(X_1+X_3+X_5),
\]
\begin{align*}
F_3=\,&X_0^2+X_2^2+X_4^2-X_0X_2-X_2X_4-X_0X_4\cr
&-(X_1^2+X_3^2+X_5^2-X_1X_3-X_3X_5-X_1X_5).
\end{align*}
By Lemma~\ref{L5.2}, 
\[
E_2=\prod_{\phi\in\mathcal E_{6,2}}\phi(F_2)=\prod_{(i_1,\dots,i_5)}(X_0+X_{i_2}+X_{i_4}-X_{i_1}-X_{i_3}-X_{i_5})
\]
and 
\[
E_3=\prod_{\phi\in\mathcal E_{6,2}}\phi(F_3)=\prod_{(i_1,\dots,i_5)}F_3(X_0,X_{i_1},\dots,X_{i_5}),
\]
where the range of $(i_1,\dots,i_5)$ is given in \eqref{i1-i5}. The symmetric reductions of $E_2$ and $E_3$ are
\[
\longequation{
H_2=s_1^{10}-12  s_1^8  s_2+24  s_1^7  s_3+48  s_1^6
    s_2^2-16  s_1^6  s_4-192  s_1^5  s_2  s_3-32
    s_1^5  s_5-64  s_1^4  s_2^3+128  s_1^4  s_2
    s_4+192  s_1^4  s_3^2-320  s_1^4  s_6+384  s_1^3
    s_2^2  s_3+256  s_1^3  s_2  s_5-256  s_1^3  s_3
    s_4-256  s_1^2  s_2^2  s_4-768  s_1^2  s_2
    s_3^2+1536  s_1^2  s_2  s_6-256  s_1^2  s_3
    s_5-512  s_1  s_2^2  s_5+1024  s_1  s_2  s_3
    s_4+512  s_1  s_3^3-2048  s_1  s_3  s_6-1024
    s_2^2  s_6+1024  s_2  s_3  s_5-1024  s_3^2
    s_4+4096  s_4  s_6-1024  s_5^2
}
\]
and
\[
\longequation{
H_3=
   s_1^{20}-24  s_2  s_1^{18}+33  s_3  s_1^{17}+246  s_2^2
    s_1^{16}-7  s_4  s_1^{16}-669  s_2  s_3  s_1^{15}+46
    s_5  s_1^{15}-1396  s_2^3  s_1^{14}+444  s_3^2
    s_1^{14}+116  s_2  s_4  s_1^{14}-125  s_6
    s_1^{14}+5631  s_2^2  s_3  s_1^{13}-127  s_3  s_4
    s_1^{13}-803  s_2  s_5  s_1^{13}+4737  s_2^4
    s_1^{12}-7383  s_2  s_3^2  s_1^{12}-66  s_4^2
    s_1^{12}-763  s_2^2  s_4  s_1^{12}+1151  s_3  s_5
    s_1^{12}+2100  s_2  s_6  s_1^{12}+3113  s_3^3
    s_1^{11}-25191  s_2^3  s_3  s_1^{11}+1600  s_2  s_3
    s_4  s_1^{11}+5593  s_2^2  s_5  s_1^{11}-129  s_4
    s_5  s_1^{11}-2900  s_3  s_6  s_1^{11}-9612  s_2^5
    s_1^{10}+48942  s_2^2  s_3^2  s_1^{10}+1002  s_2  s_4^2
    s_1^{10}+356  s_5^2  s_1^{10}+2490  s_2^3  s_4
    s_1^{10}-631  s_3^2  s_4  s_1^{10}-15917  s_2  s_3
    s_5  s_1^{10}-13960  s_2^2  s_6  s_1^{10}+1000  s_4
    s_6  s_1^{10}-40713  s_2  s_3^3  s_1^9-1779  s_3
    s_4^2  s_1^9+63180  s_2^4  s_3  s_1^9-7473  s_2^2
    s_3  s_4  s_1^9-19434  s_2^3  s_5  s_1^9+11172
    s_3^2  s_5  s_1^9+1218  s_2  s_4  s_5  s_1^9+37680
    s_2  s_3  s_6  s_1^9-3000  s_5  s_6  s_1^9+10800
    s_2^6  s_1^8+11988  s_3^4  s_1^8+477  s_4^3
    s_1^8-161703  s_2^3  s_3^2  s_1^8-5544  s_2^2  s_4^2
    s_1^8-3672  s_2  s_5^2  s_1^8-4032  s_2^4  s_4
    s_1^8+5238  s_2  s_3^2  s_4  s_1^8+82431  s_2^2
    s_3  s_5  s_1^8-621  s_3  s_4  s_5  s_1^8+45792
    s_2^3  s_6  s_1^8-25110  s_3^2  s_6  s_1^8-10800
    s_2  s_4  s_6  s_1^8+199098  s_2^2  s_3^3
    s_1^7+18603  s_2  s_3  s_4^2  s_1^7+4374  s_3
    s_5^2  s_1^7-84240  s_2^5  s_3  s_1^7+567  s_3^3
    s_4  s_1^7+15336  s_2^3  s_3  s_4  s_1^7+33696
    s_2^4  s_5  s_1^7-115263  s_2  s_3^2  s_5
    s_1^7-2943  s_4^2  s_5  s_1^7-3780  s_2^2  s_4
    s_5  s_1^7-180144  s_2^2  s_3  s_6  s_1^7+16200
    s_3  s_4  s_6  s_1^7+32400  s_2  s_5  s_6
    s_1^7-5184  s_2^7  s_1^6-115425  s_2  s_3^4  s_1^6-3402
    s_2  s_4^3  s_1^6+266328  s_2^4  s_3^2  s_1^6+13284
    s_2^3  s_4^2  s_1^6-15552  s_3^2  s_4^2  s_1^6+12636
    s_2^2  s_5^2  s_1^6+4779  s_4  s_5^2  s_1^6+2592
    s_2^5  s_4  s_1^6-13932  s_2^2  s_3^2  s_4
    s_1^6+52974  s_3^3  s_5  s_1^6-189540  s_2^3  s_3
    s_5  s_1^6+2754  s_2  s_3  s_4  s_5  s_1^6-73872
    s_2^4  s_6  s_1^6+231336  s_2  s_3^2  s_6
    s_1^6+38880  s_2^2  s_4  s_6  s_1^6-48600  s_3
    s_5  s_6  s_1^6+24057  s_3^5  s_1^5-431568  s_2^3
    s_3^3  s_1^5+5589  s_3  s_4^3  s_1^5-729  s_5^3
    s_1^5-63180  s_2^2  s_3  s_4^2  s_1^5-29889  s_2
    s_3  s_5^2  s_1^5+46656  s_2^6  s_3  s_1^5-7290
    s_2  s_3^3  s_4  s_1^5-11664  s_2^4  s_3  s_4
    s_1^5-23328  s_2^5  s_5  s_1^5+396576  s_2^2  s_3^2
    s_5  s_1^5+20898  s_2  s_4^2  s_5  s_1^5+3888
    s_2^3  s_4  s_5  s_1^5+5103  s_3^2  s_4  s_5
    s_1^5-96228  s_3^3  s_6  s_1^5+373248  s_2^3  s_3
    s_6  s_1^5-116640  s_2  s_3  s_4  s_6
    s_1^5-116640  s_2^2  s_5  s_6  s_1^5+369603  s_2^2
    s_3^4  s_1^4-729  s_4^4  s_1^4+5832  s_2^2  s_4^3
    s_1^4-174960  s_2^5  s_3^2  s_1^4-11664  s_2^4  s_4^2
    s_1^4+99873  s_2  s_3^2  s_4^2  s_1^4-14580  s_2^3
    s_5^2  s_1^4+15309  s_3^2  s_5^2  s_1^4-33534  s_2
    s_4  s_5^2  s_1^4+10935  s_3^4  s_4  s_1^4+11664
    s_2^3  s_3^2  s_4  s_1^4-365229  s_2  s_3^3  s_5
    s_1^4-35721  s_3  s_4^2  s_5  s_1^4+163296  s_2^4
    s_3  s_5  s_1^4-2916  s_2^2  s_3  s_4  s_5
    s_1^4+46656  s_2^5  s_6  s_1^4-682344  s_2^2  s_3^2
    s_6  s_1^4-46656  s_2^3  s_4  s_6  s_1^4+87480
    s_3^2  s_4  s_6  s_1^4+349920  s_2  s_3  s_5
    s_6  s_1^4-150903  s_2  s_3^5  s_1^3+349920  s_2^4
    s_3^3  s_1^3-
    17496  s_2  s_3  s_4^3  s_1^3+4374
    s_2  s_5^3  s_1^3-52488  s_3^3  s_4^2  s_1^3+69984
    s_2^3  s_3  s_4^2  s_1^3+52488  s_2^2  s_3  s_5^2
    s_1^3+63423  s_3  s_4  s_5^2  s_1^3+17496  s_2^2
    s_3^3  s_4  s_1^3+124659  s_3^4  s_5  s_1^3+8748
    s_4^3  s_5  s_1^3-454896  s_2^3  s_3^2  s_5
    s_1^3-34992  s_2^2  s_4^2  s_5  s_1^3-17496  s_2
    s_3^2  s_4  s_5  s_1^3+524880  s_2  s_3^3  s_6
    s_1^3-279936  s_2^4  s_3  s_6  s_1^3+        
    209952  s_2^2
    s_3  s_4  s_6  s_1^3+139968  s_2^3  s_5  s_6
    s_1^3-262440  s_3^2  s_5  s_6  s_1^3+19683  s_3^6
    s_1^2-393660  s_2^3  s_3^4  s_1^2+13122  s_3^2  s_4^3
    s_1^2-19683  s_3  s_5^3  s_1^2-157464  s_2^2  s_3^2
    s_4^2  s_1^2-59049  s_2  s_3^2  s_5^2  s_1^2-39366
    s_4^2  s_5^2  s_1^2+52488  s_2^2  s_4  s_5^2
    s_1^2-39366  s_2  s_3^4  s_4  s_1^2+629856  s_2^2
    s_3^3  s_5  s_1^2+104976  s_2  s_3  s_4^2  s_5
    s_1^2+19683  s_3^3  s_4  s_5  s_1^2-137781  s_3^4
    s_6  s_1^2+629856  s_2^3  s_3^2  s_6  s_1^2-314928
    s_2  s_3^2  s_4  s_6  s_1^2-629856  s_2^2  s_3
    s_5  s_6  s_1^2+236196  s_2^2  s_3^5  s_1+78732
    s_4  s_5^3  s_1+157464  s_2  s_3^3  s_4^2
    s_1+19683  s_3^3  s_5^2  s_1-157464  s_2  s_3  s_4
    s_5^2  s_1+19683  s_3^5  s_4  s_1-433026  s_2
    s_3^4  s_5  s_1-78732  s_3^2  s_4^2  s_5
    s_1-629856  s_2^2  s_3^3  s_6  s_1+157464  s_3^3
    s_4  s_6  s_1+944784  s_2  s_3^2  s_5  s_6
    s_1-59049  s_2  s_3^6-59049  s_5^4-59049  s_3^4
    s_4^2+118098  s_3^2  s_4  s_5^2+118098  s_3^5
    s_5+236196  s_2  s_3^4  s_6-472392  s_3^3  s_5
    s_6.
}
\]

\medskip
\noindent{\bf Case 3.} $G=D_4=\langle a,b\mid a^4=b^2=1,\ b^{-1}ab=a^{-1}\rangle$.

\begin{table}[h]
\caption{Character table of $D_4$}\label{Tb3}
\vspace{-2em}
\[
\renewcommand*{\arraystretch}{1.2}
\begin{tabular}{c|ccccc}
\hline
conj. cls. rep. & $1$ & $a^2$ & $a$ & $b$ & $ab$  \\ 
conj. cls. size & 1 & 1& 2& 2& 2\\ \hline
$\chi_1$ & $1$ & $1$ & $1$ & $1$ & $1$ \\
$\chi_2$ & $1$ & $1$ & $1$ & $-1$ & $-1$  \\
$\chi_3$ & $1$ & $1$ & $-1$ & $1$ & $-1$  \\
$\chi_4$ & $1$ & $1$ & $-1$ & $-1$ & $1$  \\
$\chi_5$ & $2$ & $-2$ & $0$ & $0$ & $0$  \\ \hline
\end{tabular}
\]
\end{table}
We have
\[
F_2=\sum_{g\in\langle a\rangle}X_g-\sum_{g\in\langle a\rangle b}X_g,
\]
\[
F_3=\sum_{g\in\langle a^2,b\rangle}X_g-\sum_{g\in a\langle a^2,b\rangle}X_g,
\]
\[
F_4=\sum_{g\in\langle a^2,ab\rangle}X_g-\sum_{g\in a\langle a^2,ab\rangle}X_g,
\]
\[
F_5=\sum_{g\in\langle a\rangle}X_g^2-\sum_{g\in\langle a\rangle b}X_g^2-2X_1X_{a^2}-2X_aX_{a^3}+2X_bX_{a^2b}+2X_{ab}X_{a^3b}.
\]
$F_2,F_3,F_3$ are in the same $S_8$-orbit and hence have the same symmetrization. Rename the elements of $G$ as
\[
(1,a,a^2,a^3,b,ab,a^2b,a^3b)\to (0,2,4,6,1,3,5,7).
\]
Then 
\[
F_2=X_0+X_2+X_4+X_6-(X_1+X_3+X_5+X_7),
\]
\[
F_5=(X_0-X_4)^2+(X_2-X_6)^2-(X_1-X_5)^2-(X_3-X_7)^2.
\]
By Lemma~\ref{L5.2},
\[
E_2=\prod_{\phi\in\mathcal E_{8,2}}\phi(F_2).
\]
Let $B_i=(X_i-X_{i+4})^2$, $0\le i\le 3$. Then $F_5=B_0+B_2-B_1-B_3$. Let 
\[
F_5'=(B_0+B_2-B_1-B_3)(B_1+B_2-B_0-B_3)(B_3+B_2-B_1-B_0),
\]
which is symmetric in $B_0,B_1,B_2,B_3$. By Lemma~\ref{L5.2},
\[
E_5:=\prod_{\phi\in\mathcal E_{8,4}}\phi(F_5')
\]
is a symmetrization of $F_5$.

$E_2$ is a product of 35 homogeneous linear polynomials in 8 variables and $E_5$ is a product of 315 homogeneous quadratic polynomials in 8 variables. The computation of $H_2$ (the symmetric reduction of $E_2$) is impractical and the computation of $H_5$ (the symmetric reduction of $E_5$) is impossible.

\medskip
\noindent{\bf Case 4.} $G=Q_8=\langle a,b\mid a^4=1,\ b^2=a^2,\ b^{-1}ab=a^{-1}\rangle$.

\begin{table}[h]
\caption{Character table of $Q_8$}\label{Tb4}
\vspace{-2em}
\[
\renewcommand*{\arraystretch}{1.2}
\begin{tabular}{c|ccccc}
\hline
conj. cls. rep. & $1$ & $a^2$ & $a$ & $b$ & $ab$  \\ 
conj. cls. size & 1 & 1& 2& 2& 2\\ \hline
$\chi_1$ & $1$ & $1$ & $1$ & $1$ & $1$ \\
$\chi_2$ & $1$ & $1$ & $1$ & $-1$ & $-1$  \\
$\chi_3$ & $1$ & $1$ & $-1$ & $1$ & $-1$  \\
$\chi_4$ & $1$ & $1$ & $-1$ & $-1$ & $1$  \\
$\chi_5$ & $2$ & $-2$ & $0$ & $0$ & $0$  \\ \hline
\end{tabular}
\]
\end{table}
We have
\[
F_2=\sum_{g\in\langle a\rangle}X_g-\sum_{g\in\langle a\rangle b}X_g,
\]
\[
F_3=\sum_{g\in\langle a^2,b\rangle}X_g-\sum_{g\in a\langle a^2,b\rangle}X_g,
\]
\[
F_4=\sum_{g\in\langle a^2,ab\rangle}X_g-\sum_{g\in a\langle a^2,ab\rangle}X_g,
\]
\[
F_5=(X_1-X_{a^2})^2+(X_a-X_{a^3})^2+(X_b-X_{a^2b})^2+(X_{ab}-X_{a^3b})^2.
\]
As in Case 3, rename the elements of $G$ as
\[
(1,a,a^2,a^3,b,ab,a^2b,a^3b)\to (0,2,4,6,1,3,5,7).
\]
$F_2,F_3,F_4$ are in the same $S_8$-orbit and $E_2$ is the same as in Case 3. We have
\[
F_5=(X_0-X_4)^2+(X_1-X_5)^2+(X_2-X_6)^2+(X_3-X_7)^2.
\]
By Lemma~\ref{L5.2} (with $B_i=(X_i-X_{i+4})^2$, $0\le i\le 3$),
\[
E_5:=\prod_{\phi\in\mathcal E_{8,4}}\phi(F_5)
\]
is a symmetrization of $F_5$. This is a product of 105 homogeneous quadratic polynomials in 8 variable. We acre unable to compute the symmetric reductions $H_2$ and $H_5$ of $E_2$ and $E_5$.

\medskip
\noindent{\bf Case 5.} $G=(\Bbb Z/2\Bbb Z)^3$.

\begin{table}[h]
\caption{Character table of $(\Bbb Z/2\Bbb Z)^3$}\label{Tb5}
\vspace{-2em}
\[
\renewcommand*{\arraystretch}{1.2}
\begin{tabular}{c|cccccccc}
\hline
conj. cls. rep. & $(000)$ & $(001)$ & $(010)$ & $(011)$ & $(100)$ & $(101)$ & $(110)$ & $(111)$ \\ 
conj. cls. size & 1 & 1&  1 & 1 & 1 & 1 & 1 & 1\\ \hline
$\chi_1$ & $1$ & $1$ & $1$ & $1$ & $1$ & $1$ & $1$ & $1$ \\
$\chi_2$ & $1$ & $-1$ & $1$ & $-1$ & $1$ & $-1$ & $1$ & $-1$ \\
$\chi_3$ & $1$ & $1$ & $-1$ & $-1$ & $1$ & $1$ & $-1$ & $-1$ \\ 
$\chi_4$ & $1$ & $-1$ & $-1$ & $1$ & $1$ & $-1$ & $-1$ & $1$ \\ 
$\chi_5$ & $1$ & $1$ & $1$ & $1$ & $-1$ & $-1$ & $-1$ & $-1$ \\
$\chi_6$ & $1$ & $-1$ & $1$ & $-1$ & $-1$ & $1$ & $-1$ & $1$ \\
$\chi_7$ & $1$ & $1$ & $-1$ & $-1$ & $-1$ & $-1$ & $1$ & $1$ \\ 
$\chi_8$ & $1$ & $-1$ & $-1$ & $1$ & $-1$ & $1$ & $1$ & $-1$ \\ 
\hline
\end{tabular}
\]
\end{table}
In this case, $F_2,\dots,F_8$ are in the same $S_8$-orbit. $F_2$ and $E_2$ are the same as those in Case~3.

\medskip
\noindent{\bf Case 6.} $G=(\Bbb Z/2\Bbb Z)\times(\Bbb Z/4\Bbb Z)$.

\begin{table}[h]
\caption{Character table of $(\Bbb Z/2\Bbb Z)\times(\Bbb Z/4\Bbb Z)$}\label{Tb6}
\vspace{-2em}
\[
\renewcommand*{\arraystretch}{1.2}
\begin{tabular}{c|cccccccc}
\hline
conj. cls. rep. & $(00)$ & $(01)$ & $(02)$ & $(03)$ & $(10)$ & $(11)$ & $(12)$ & $(13)$ \\ 
conj. cls. size & 1 & 1&  1 & 1 & 1 & 1 & 1 & 1\\ \hline
$\chi_1$ & $1$ & $1$ & $1$ & $1$ & $1$ & $1$ & $1$ & $1$ \\
$\chi_2$ & $1$ & $-1$ & $1$ & $-1$ & $1$ & $-1$ & $1$ & $-1$ \\
$\chi_3$ & $1$ & $i$ & $-1$ & $-i$ & $1$ & $i$ & $-1$ & $-i$ \\ 
$\chi_4$ & $1$ & $-i$ & $-1$ & $i$ & $1$ & $-i$ & $-1$ & $i$ \\ 
$\chi_5$ & $1$ & $1$ & $1$ & $1$ & $-1$ & $-1$ & $-1$ & $-1$ \\
$\chi_6$ & $1$ & $-1$ & $1$ & $-1$ & $-1$ & $1$ & $-1$ & $1$ \\
$\chi_7$ & $1$ & $i$ & $-1$ & $-i$ & $-1$ & $-i$ & $1$ & $i$ \\ 
$\chi_8$ & $1$ & $-i$ & $-1$ & $i$ & $-1$ & $i$ & $1$ & $-i$ \\ 
\hline
\end{tabular}
\]
\end{table}
In this case, $F_2,F_5,F_6$ are in the same $S_8$-orbit, $F_3,F_4,F_7,F_8$ are in the same $S_8$-orbit, and $F_2$ and $E_2$ are the same as those in Case~3.  Rename the element $(a,b)\in G$ as $4a+b$ and let $B_i=X_i+X_{i+4}$, $0\le i\le 3$. Then 
\[
F_3F_4=(B_0-B_2)^2+(B_1-B_3)^2.
\]
Let 
\[
F_3'=\bigl((B_0-B_2)^2+(B_1-B_3)^2\bigr)\bigl((B_1-B_2)^2+(B_0-B_3)^2\bigr)\bigl((B_3-B_2)^2+(B_1-B_0)^2\bigr),
\]
which is symmetric in $B_0,B_1,B_2,B_3$. By Lemma~\ref{L5.2},
\[
E_5:=\prod_{\phi\in\mathcal E_{8,4}}\phi(F_3')
\]
is a symmetrization of $F_3$. Its symmetric reduction is beyond reach.

\section*{Appendix}

\noindent A1. The expression of $\Theta_{7,7}$.

\medskip

Recall from \eqref{Psi} that 
\[
\Psi_7(X_0,\dots,X_6)=\prod_{i\in(\Bbb Z/7\Bbb Z)^\times}\Bigl(\sum_{j\in\Bbb Z/7\Bbb Z}\epsilon_7^{ij}X_j\Bigr)\in\Bbb Z[X_0,\dots,X_6].
\]
The expansion of $\Psi_7$ is lengthy and is not included here. However, it can be easily generated by the following Mathematica code:

\medskip

\begin{mmaCell}{Code}
Array [x, 7, 0];
psi7 = 1;
For[i = 1, i < 7, i++,
  psi7 = psi7*Sum[e^(i*j) x[j], {j, 0, 6}];
  ];
psi7 = PolynomialMod[psi7, Cyclotomic[7, e]];
\end{mmaCell}

\medskip

\noindent
$\Theta_{7,7}$ is the symmetrization of $\Psi_7$. By equation~(2.8) of \cite{Hou-preprint},
\[
\Theta_{7,7}=\prod_{(i_2,\dots,i_6)}\Psi_7(X_0,X_1,X_{i_2},X_{i_3},X_{i_4},X_{i_5},X_{i_6}),
\]
where $(i_2,\dots,i_6)$ runs through all permutations of $2,\dots,6$.

\medskip
\noindent A2. $\theta_{6,6}(1,0,\dots,0,a)$ in characteristic $p$, $5\le p\le 97$.
{\small
\[
\renewcommand*{\arraystretch}{1.4}
\begin{tabular}{c|l}
\hline
$p$ & \hfil$\theta_{6,6}(1,0,\dots,0,a)$\hfil \\ \hline
5 & $3(a^2+4a+2)^5$ \\
7 & $3 (a^{14}+6 a^{13}+2 a^{12}+5 a^{11}+6 a^{10}+5 a^8+6 a^7+a^5+6 a^4+2 a^2+5
   a+5) $\\
11 & $4 (a^{19}+3 a^{18}+10 a^{17}+5 a^{16}+7 a^{15}+4 a^{14}+4 a^{13}+6 a^{12}+8
   a^{11}+10 a^{10} $\\
   & $+4 a^9+10 a^8+7 a^7+5 a^6+5 a^5+8 a^4+2 a^3+a^2+2 a+3) $\\
13& $7 (a+12)  (a^{16}+4 a^{15}+7 a^{14}+5 a^{13}+2 a^{12}+5 a^{11}+9 a^{10}+5 a^9$\\
&$+8 a^8+12 a^7+a^6+11 a^5+8 a^4+10 a^3+3 a^2+10 a+11  )$\\

17& $10  (a^3+4 a^2+13 a+8  )  (a^3+7 a^2+11 a+8  )  (a^3+14 a^2+4 a+1  )$\\
& $(a^{10}+14 a^9+2 a^8+6 a^7+6 a^6+2 a^5+2 a^4+5 a^3+7 a^2+8 a+14  )$\\

19& $4 (a+15)  (a^7+17 a^6+6 a^5+10 a^4+8 a^3+9 a^2+7 a+17  )$\\
&$  (a^{11}+12 a^{10}+18 a^9+4 a^8+3 a^7+17 a^6+17 a^5+12 a^4+a^3+18 a^2+3  )$\\

23& $11  (a^2+12 a+3  )  (a^5+4 a^4+19 a^3+18 a^2+15 a+14  )(a^{12}+12 a^{11}+18 a^{10}+10 a^9$\\
&$+10 a^8+18 a^7+5 a^6+2 a^5+3 a^4+22 a^3+4 a^2+19 a+12  )$\\

29& $28  (a^{19}+13 a^{18}+10 a^{17}+15 a^{16}+20 a^{15}+18 a^{14}+13 a^{13}+a^{12}+23 a^{11}+22 a^{10}$\\
&$+10 a^9+20 a^8+2 a^7+11 a^6+21 a^5+24 a^4+25 a^3+7 a^2+19 a+28  )$\\

31& $(a+19)  (a^5+5 a^4+18 a^3+13 a^2+4 a+16  )  (a^5+18 a^4+12 a^3+7 a^2+9 a+9  ) $\\ &$(a^8+27 a^7+6 a^6+27 a^5+25 a^4+20 a^3+9 a^2+21 a+4  )$\\

37& $5  (a^2+7 a+9  )  (a^4+2 a^3+35 a^2+1  )  (a^{13}+4 a^{12}+31 a^{11}+4 a^{10}+26 a^9+26 a^8$\\
&$+5 a^7+5 a^6+3 a^5+a^4+16 a^3+33 a^2+36 a+14  )$\\

41& $32  (a^3+33 a^2+9 a+31  )  (a^6+2 a^5+27 a^4+24 a^3+9 a^2+8 a+19  )$\\
&$  (a^{10}+8 a^9+17 a^8+28 a^7+33 a^6+23 a^5+6 a^4+37 a^3+26 a^2+39 a+17  )$\\

43& $42  (a^2+a+21  )  (a^2+17 a+35  )  (a^2+37 a+18  )  (a^3+7 a^2+21 a+35  )$\\
&$  (a^4+39 a^3+5 a^2+24 a+30  )  (a^6+4 a^5+28 a^4+9 a^3+34 a^2+9 a+7  )$\\

47& $15  (a^3+19 a^2+13 a+8  )  (a^{16}+41 a^{15}+8 a^{14}+9 a^{13}+21 a^{12}+12 a^{11}+10 a^{10}$\\
&$+17 a^9+45 a^8+42 a^7+43 a^5+16 a^4+32 a^3+36 a^2+3 a+38  )$\\

53& $50 (a+6)^2 (a+25)^2  (a^2+46 a+20  )  (a^4+4 a^3+32 a^2+6 a+40  )$\\
&$  (a^4+37 a^3+9 a^2+7 a+43  )  (a^5+19 a^4+3 a^3+12 a^2+20 a+49  )$\\

59& $20  (a^{19}+47 a^{18}+45 a^{17}+51 a^{16}+53 a^{15}+25 a^{14}+51 a^{13}+31 a^{12}+58 a^{11}+18 a^{10}$\\
&$+52 a^9+42 a^8+15 a^7+10 a^6+22 a^5+16 a^4+38 a^3+3 a^2+49 a+3  )$\\

61& $47 (a+33)^2  (a^3+36 a^2+29 a+35  )  (a^3+57 a^2+43 a+24  )$\\
&$  (a^4+27 a^3+3 a^2+12 a+53  )  (a^7+9 a^6+9 a^5+9 a^4+33 a^3+24 a^2+54 a+32  )$\\

67& $32 (a+36)  (a^4+2 a^3+30 a^2+12 a+5  )  (a^5+37 a^4+31 a^3+48 a^2+20 a+60  )$\\  
&$(a^9+31 a^8+13 a^7+57 a^6+63 a^5+35 a^4+50 a^3+39 a^2+6 a+24  )$\\

71& $57  (a^3+17 a^2+70 a+48  )  (a^{16}+67 a^{15}+26 a^{14}+12 a^{13}+5 a^{12}+47 a^{11}+60 a^{10}$\\
&$+69 a^9+29 a^8+14 a^7+13 a^6+25 a^5+65 a^4+15 a^3+67 a^2+49 a+43  )$\\
\hline
\end{tabular}
\]
}
{\small
\[
\renewcommand*{\arraystretch}{1.4}
\begin{tabular}{c|l}
\hline
$p$ & \hfil$\theta_{6,6}(1,0,\dots,0,a)$\hfil \\ \hline

73& $62  (a^3+63 a^2+39 a+6  )  (a^5+6 a^4+14 a^3+57 a^2+62 a+58  ) $\\
&$ (a^{11}+12 a^{10}+40 a^9+32 a^8+39 a^7+69 a^6+68 a^5+13 a^4+38 a^3+6 a^2+59 a+57  )$\\

79& $40 (a+29) (a+45)  (a^2+26 a+27  )  (a^2+61 a+75  )  (a^{13}+71 a^{12}+44 a^{11}+41 a^{10}$\\
&$+47 a^9+2 a^8+2 a^7+53 a^6+48 a^5+52 a^4+70 a^3+46 a^2+64 a+39  )$\\

83& $82  (a^{19}+33 a^{18}+37 a^{17}+12 a^{16}+12 a^{15}+75 a^{14}+3 a^{13}+a^{12}+24 a^{11}+65 a^{10}$\\
&$+45 a^9+57 a^8+35 a^7+14 a^6+49 a^5+27 a^4+8 a^3+28 a^2+23 a+82  )$\\

89& $55 (a+22)^2 (a+40) (a+78)  (a^{15}+43 a^{14}+51 a^{13}+6 a^{12}+29 a^{11}+82 a^{10}+14 a^9$\\
&$+13 a^8+53 a^7+43 a^6+22 a^5+21 a^4+2 a^3+33 a^2+79 a+2  )$\\

97& $10  (a^3+37 a^2+a+7  )  (a^5+26 a^4+48 a^3+53 a^2+36 a+32  )$\\
&$  (a^5+73 a^4+78 a^3+85 a^2+33 a+30  )  (a^6+45 a^5+15 a^4+89 a^3+61 a^2+16 a+60  )$\\      
\hline
\end{tabular}
\]
}

\medskip
\noindent A3. Values of $\text{Irr}\,(q,n)$, $\text{Norm}\,(q,n)$, $\text{Suff}\,(q,n)$,  $n=6,7$, $q\le 49$, $\text{gcd}(q,n)=1$.

\medskip

Refer to \eqref{Irr}, \eqref{Norm}, \eqref{Suff} for the definitions of $\text{Irr}\,(q,n)$, $\text{Norm}\,(q,n)$, $\text{Suff}\,(q,n)$.
\[
\renewcommand*{\arraystretch}{1.4}
\begin{tabular}{c|c|c|c}
\multicolumn{4}{c}{$n=6$}\\
\hline
$q$ & $\text{Irr}\,(q,6)$ &$\text{Norm}\,(q,6)$ & $\text{Suff}\,(q,6)$\\ \hline
5 & 1 & 1& 1\\
7 & 0 & 0& 0\\
11 & 2 & 2& 2\\
13 & 2 & 1& 1\\
17 & 4 & 3& 3\\
19 & 3 & 2& 2\\
23 & 2 & 2& 2\\
25 & 4 & 4& 4\\
29 & 7 & 7& 7\\
31 & 4 & 3& 3\\
37 & 6 & 6& 6\\
41 & 6 & 6& 6\\
43 & 5 & 5& 5\\
47 & 10 & 10& 10\\
49 & 10 & 10& 10\\
\hline
\end{tabular}
\]

\[
\renewcommand*{\arraystretch}{1.4}
\begin{tabular}{c|c|c|c}
\multicolumn{4}{c}{$n=7$}\\
\hline
$q$ & $\text{Irr}\,(q,7)$ &$\text{Norm}\,(q,7)$ & $\text{Suff}\,(q,7)$\\ \hline
2 & 1 & 1& 1\\
3 & 0 & 0& 0\\
4 & 1 & 1& 1\\
5 & 2 & 2& 2\\
8 & 1 & 1& 1\\
9 & 2 & 2& 2\\
11 & 3 & 3& 3\\
13 & 1 & 1& 1\\
16 & 5 & 5& 5\\
17 & 2 & 2& 2\\
19 & 2 & 2& 2\\
23 & 1 & 1& 1\\
25 & 2 & 2& 2\\
27 & 0 & 0& 0\\
29 & 6 & 5& 5\\
31 & 4 & 4& 4\\
32 & 11 & 11& 11\\
37 & 5 & 5& 5\\
41 & 4 & 4& 4\\
43 & 4 & 4& 4\\
47 & 7 & 7& 7\\
\hline
\end{tabular}
\]

\medskip
\noindent A4. Intermediate data in the computation of $\theta_{7,7}(1,0,\dots,0,a)$ with $p=2$.

\medskip
\noindent Minimal polynomials of $u_i$:

\[
\renewcommand*{\arraystretch}{1.4}
\begin{tabular}{c|c|c}
\hline
$i$ &$k_i$ & $g_i(X)$\\ \hline
$1$& $7$& $X^7+X^6+1$\\
$2$& $7$& $X^7+X^4+1$\\
$3$& $8$& $X^8+X^7+X^5+X^4+1$\\
$4$& $8$& $X^8+X^6+X^5+X^4+1$\\
$5$& $7$& $X^7+X^6+X^5+X^4+1$\\
$6$& $7$& $X^7+X^3+1$\\
$7$& $4$& $X^4+X^3+1$\\
$8$& $8$& $X^8+X^6+X^5+X^2+1$\\
$9$& $6$& $X^6+X^5+X^4+X^2+1$\\
$10$& $8$& $X^8+X^7+X^3+X^2+1$\\
$11$& $6$& $X^6+X+1$\\
$12$& $8$& $X^8+X^7+X^6+X^5+X^4+X+1$\\
$13$& $8$& $X^8+X^5+X^3+X+1$\\
$14$& $6$& $X^6+X^4+X^3+X+1$\\
$15$& $8$& $X^8+X^7+X^6+X^4+X^2+X+1$\\
\hline
\end{tabular}
\]

\vskip2em
\medskip
\noindent The matrices $[A(u_i)\ C(u_i)]$:

\vskip1em
\medskip
\noindent $i=1$.

{\tiny
\begin{align*}
1  0  0  0  0  0  0  1  1  1  1  1  1  1  0  1  0  1  0  1  0  0  1  1  0  0  1  1  1  0  1  1  1  0  1  0  0  1  0  1  1  0  0  0  1  1  0  1  1  1  1  0  1  1  0  1  0  1  1  0  1  1  0  0  1  0  0  1  0  0  0  1  1  1  0  0  0  0  1  0  1  1  1  1  1  0  0  1  0  1  0  1  1  1  0  0  1  1  0  1  0  0  0  &\  0 \cr
 0  1  0  0  0  0  0  0  1  1  1  1  1  1  1  0  1  0  1  0  1  0  0  1  1  0  0  1  1  1  0  1  1  1  0  1  0  0  1  0  1  1  0  0  0  1  1  0  1  1  1  1  0  1  1  0  1  0  1  1  0  1  1  0  0  1  0  0  1  0  0  0  1  1  1  0  0  0  0  1  0  1  1  1  1  1  0  0  1  0  1  0  1  1  1  0  0  1  1  0  1  0  0  &\  1 \cr
 0  0  1  0  0  0  0  0  0  1  1  1  1  1  1  1  0  1  0  1  0  1  0  0  1  1  0  0  1  1  1  0  1  1  1  0  1  0  0  1  0  1  1  0  0  0  1  1  0  1  1  1  1  0  1  1  0  1  0  1  1  0  1  1  0  0  1  0  0  1  0  0  0  1  1  1  0  0  0  0  1  0  1  1  1  1  1  0  0  1  0  1  0  1  1  1  0  0  1  1  0  1  0  &\  1 \cr
 0  0  0  1  0  0  0  0  0  0  1  1  1  1  1  1  1  0  1  0  1  0  1  0  0  1  1  0  0  1  1  1  0  1  1  1  0  1  0  0  1  0  1  1  0  0  0  1  1  0  1  1  1  1  0  1  1  0  1  0  1  1  0  1  1  0  0  1  0  0  1  0  0  0  1  1  1  0  0  0  0  1  0  1  1  1  1  1  0  0  1  0  1  0  1  1  1  0  0  1  1  0  1  &\  0 \cr
 0  0  0  0  1  0  0  0  0  0  0  1  1  1  1  1  1  1  0  1  0  1  0  1  0  0  1  1  0  0  1  1  1  0  1  1  1  0  1  0  0  1  0  1  1  0  0  0  1  1  0  1  1  1  1  0  1  1  0  1  0  1  1  0  1  1  0  0  1  0  0  1  0  0  0  1  1  1  0  0  0  0  1  0  1  1  1  1  1  0  0  1  0  1  0  1  1  1  0  0  1  1  0  &\  0 \cr
 0  0  0  0  0  1  0  0  0  0  0  0  1  1  1  1  1  1  1  0  1  0  1  0  1  0  0  1  1  0  0  1  1  1  0  1  1  1  0  1  0  0  1  0  1  1  0  0  0  1  1  0  1  1  1  1  0  1  1  0  1  0  1  1  0  1  1  0  0  1  0  0  1  0  0  0  1  1  1  0  0  0  0  1  0  1  1  1  1  1  0  0  1  0  1  0  1  1  1  0  0  1  1  &\  0 \cr
 0  0  0  0  0  0  1  1  1  1  1  1  1  0  1  0  1  0  1  0  0  1  1  0  0  1  1  1  0  1  1  1  0  1  0  0  1  0  1  1  0  0  0  1  1  0  1  1  1  1  0  1  1  0  1  0  1  1  0  1  1  0  0  1  0  0  1  0  0  0  1  1  1  0  0  0  0  1  0  1  1  1  1  1  0  0  1  0  1  0  1  1  1  0  0  1  1  0  1  0  0  0  1  &\  1 \cr
\end{align*}
} 

\medskip
\noindent $i=2$.

{\tiny
\begin{align*}
 1  0  0  0  0  0  0  1  0  0  1  0  0  1  1  0  1  0  0  1  1  1  1  0  1  1  1  0  0  0  0  1  1  1  1  1  1  1  0  0  0  1  1  1  0  1  1  0  0  0  1  0  1  0  0  1  0  1  1  1  1  1  0  1  0  1  0  1  0  0  0  0  1  0  1  1  0  1  1  1  1  0  0  1  1  1  0  0  1  0  1  0  1  1  0  0  1  1  0  0  0  0  0  &\  1 \cr
 0  1  0  0  0  0  0  0  1  0  0  1  0  0  1  1  0  1  0  0  1  1  1  1  0  1  1  1  0  0  0  0  1  1  1  1  1  1  1  0  0  0  1  1  1  0  1  1  0  0  0  1  0  1  0  0  1  0  1  1  1  1  1  0  1  0  1  0  1  0  0  0  0  1  0  1  1  0  1  1  1  1  0  0  1  1  1  0  0  1  0  1  0  1  1  0  0  1  1  0  0  0  0  &\  0 \cr
 0  0  1  0  0  0  0  0  0  1  0  0  1  0  0  1  1  0  1  0  0  1  1  1  1  0  1  1  1  0  0  0  0  1  1  1  1  1  1  1  0  0  0  1  1  1  0  1  1  0  0  0  1  0  1  0  0  1  0  1  1  1  1  1  0  1  0  1  0  1  0  0  0  0  1  0  1  1  0  1  1  1  1  0  0  1  1  1  0  0  1  0  1  0  1  1  0  0  1  1  0  0  0  &\  1 \cr
 0  0  0  1  0  0  0  0  0  0  1  0  0  1  0  0  1  1  0  1  0  0  1  1  1  1  0  1  1  1  0  0  0  0  1  1  1  1  1  1  1  0  0  0  1  1  1  0  1  1  0  0  0  1  0  1  0  0  1  0  1  1  1  1  1  0  1  0  1  0  1  0  0  0  0  1  0  1  1  0  1  1  1  1  0  0  1  1  1  0  0  1  0  1  0  1  1  0  0  1  1  0  0  &\  1 \cr
 0  0  0  0  1  0  0  1  0  0  1  1  0  1  0  0  1  1  1  1  0  1  1  1  0  0  0  0  1  1  1  1  1  1  1  0  0  0  1  1  1  0  1  1  0  0  0  1  0  1  0  0  1  0  1  1  1  1  1  0  1  0  1  0  1  0  0  0  0  1  0  1  1  0  1  1  1  1  0  0  1  1  1  0  0  1  0  1  0  1  1  0  0  1  1  0  0  0  0  0  1  1  0  &\  1 \cr
 0  0  0  0  0  1  0  0  1  0  0  1  1  0  1  0  0  1  1  1  1  0  1  1  1  0  0  0  0  1  1  1  1  1  1  1  0  0  0  1  1  1  0  1  1  0  0  0  1  0  1  0  0  1  0  1  1  1  1  1  0  1  0  1  0  1  0  0  0  0  1  0  1  1  0  1  1  1  1  0  0  1  1  1  0  0  1  0  1  0  1  1  0  0  1  1  0  0  0  0  0  1  1  &\  1 \cr
 0  0  0  0  0  0  1  0  0  1  0  0  1  1  0  1  0  0  1  1  1  1  0  1  1  1  0  0  0  0  1  1  1  1  1  1  1  0  0  0  1  1  1  0  1  1  0  0  0  1  0  1  0  0  1  0  1  1  1  1  1  0  1  0  1  0  1  0  0  0  0  1  0  1  1  0  1  1  1  1  0  0  1  1  1  0  0  1  0  1  0  1  1  0  0  1  1  0  0  0  0  0  1  &\  1 \cr
\end{align*}
}

\medskip
\noindent $i=3$.

{\tiny
\begin{align*}
 1   0   0   0   0   0   0   0   1   1   1   0   0   0   1   1   0   0   1   0   0   1   1   0   1   1   1   0   0   1   1   1   1   0   1   1   0   0   1   1   0   1   0   0   1   0   1   1   0   0   0   1   0   0   0   0   0   0   0   1   1   1   0   0   0   1   1   0   0   1   0   0   1   1   0   1   1   1   0   0   1   1   1   1   0   1   1   0   0   1   1   0   1   0   0   1   0   1   1   0   0   0   1    &\  1  \cr
 0   1   0   0   0   0   0   0   0   1   1   1   0   0   0   1   1   0   0   1   0   0   1   1   0   1   1   1   0   0   1   1   1   1   0   1   1   0   0   1   1   0   1   0   0   1   0   1   1   0   0   0   1   0   0   0   0   0   0   0   1   1   1   0   0   0   1   1   0   0   1   0   0   1   1   0   1   1   1   0   0   1   1   1   1   0   1   1   0   0   1   1   0   1   0   0   1   0   1   1   0   0   0    &\  0  \cr
 0   0   1   0   0   0   0   0   0   0   1   1   1   0   0   0   1   1   0   0   1   0   0   1   1   0   1   1   1   0   0   1   1   1   1   0   1   1   0   0   1   1   0   1   0   0   1   0   1   1   0   0   0   1   0   0   0   0   0   0   0   1   1   1   0   0   0   1   1   0   0   1   0   0   1   1   0   1   1   1   0   0   1   1   1   1   0   1   1   0   0   1   1   0   1   0   0   1   0   1   1   0   0    &\  0  \cr
 0   0   0   1   0   0   0   0   0   0   0   1   1   1   0   0   0   1   1   0   0   1   0   0   1   1   0   1   1   1   0   0   1   1   1   1   0   1   1   0   0   1   1   0   1   0   0   1   0   1   1   0   0   0   1   0   0   0   0   0   0   0   1   1   1   0   0   0   1   1   0   0   1   0   0   1   1   0   1   1   1   0   0   1   1   1   1   0   1   1   0   0   1   1   0   1   0   0   1   0   1   1   0    &\  0  \cr
 0   0   0   0   1   0   0   0   1   1   1   0   1   1   0   1   0   0   0   1   0   1   0   0   1   0   0   0   1   0   0   1   1   1   0   0   1   0   0   0   0   1   1   1   1   1   1   1   1   0   1   0   0   0   0   1   0   0   0   1   1   1   0   1   1   0   1   0   0   0   1   0   1   0   0   1   0   0   0   1   0   0   1   1   1   0   0   1   0   0   0   0   1   1   1   1   1   1   1   1   0   1   0    &\  1  \cr
 0   0   0   0   0   1   0   0   1   0   0   1   0   1   0   1   1   0   1   0   1   1   0   0   1   0   1   0   0   0   1   1   0   1   0   1   0   1   1   1   0   1   1   1   0   1   0   0   1   1   0   0   0   0   0   0   1   0   0   1   0   0   1   0   1   0   1   1   0   1   0   1   1   0   0   1   0   1   0   0   0   1   1   0   1   0   1   0   1   1   1   0   1   1   1   0   1   0   0   1   1   0   0    &\  0  \cr
 0   0   0   0   0   0   1   0   0   1   0   0   1   0   1   0   1   1   0   1   0   1   1   0   0   1   0   1   0   0   0   1   1   0   1   0   1   0   1   1   1   0   1   1   1   0   1   0   0   1   1   0   0   0   0   0   0   1   0   0   1   0   0   1   0   1   0   1   1   0   1   0   1   1   0   0   1   0   1   0   0   0   1   1   0   1   0   1   0   1   1   1   0   1   1   1   0   1   0   0   1   1   0    &\  1  \cr
 0   0   0   0   0   0   0   1   1   1   0   0   0   1   1   0   0   1   0   0   1   1   0   1   1   1   0   0   1   1   1   1   0   1   1   0   0   1   1   0   1   0   0   1   0   1   1   0   0   0   1   0   0   0   0   0   0   0   1   1   1   0   0   0   1   1   0   0   1   0   0   1   1   0   1   1   1   0   0   1   1   1   1   0   1   1   0   0   1   1   0   1   0   0   1   0   1   1   0   0   0   1   0    &\  1  \cr
\end{align*}
} 

\medskip
\noindent $i=4$.

{\tiny
\begin{align*}
 1   0   0   0   0   0   0   0   1   0   1   1   0   0   0   1   1   1   1   0   1   0   0   0   0   1   1   1   1   1   1   1   1   0   0   1   0   0   0   0   1   0   1   0   0   1   1   1   1   1   0   1   0   1   0   1   0   1   1   1   0   0   0   0   0   1   1   0   0   0   1   0   1   0   1   1   0   0   1   1   0   0   1   0   1   1   1   1   1   1   0   1   1   1   1   0   0   1   1   0   1   1   1    &\  1  \cr
 0   1   0   0   0   0   0   0   0   1   0   1   1   0   0   0   1   1   1   1   0   1   0   0   0   0   1   1   1   1   1   1   1   1   0   0   1   0   0   0   0   1   0   1   0   0   1   1   1   1   1   0   1   0   1   0   1   0   1   1   1   0   0   0   0   0   1   1   0   0   0   1   0   1   0   1   1   0   0   1   1   0   0   1   0   1   1   1   1   1   1   0   1   1   1   1   0   0   1   1   0   1   1    &\  1  \cr
 0   0   1   0   0   0   0   0   0   0   1   0   1   1   0   0   0   1   1   1   1   0   1   0   0   0   0   1   1   1   1   1   1   1   1   0   0   1   0   0   0   0   1   0   1   0   0   1   1   1   1   1   0   1   0   1   0   1   0   1   1   1   0   0   0   0   0   1   1   0   0   0   1   0   1   0   1   1   0   0   1   1   0   0   1   0   1   1   1   1   1   1   0   1   1   1   1   0   0   1   1   0   1    &\  1  \cr
 0   0   0   1   0   0   0   0   0   0   0   1   0   1   1   0   0   0   1   1   1   1   0   1   0   0   0   0   1   1   1   1   1   1   1   1   0   0   1   0   0   0   0   1   0   1   0   0   1   1   1   1   1   0   1   0   1   0   1   0   1   1   1   0   0   0   0   0   1   1   0   0   0   1   0   1   0   1   1   0   0   1   1   0   0   1   0   1   1   1   1   1   1   0   1   1   1   1   0   0   1   1   0    &\  0  \cr
 0   0   0   0   1   0   0   0   1   0   1   1   1   0   1   0   1   1   1   1   0   1   1   0   1   1   1   1   1   0   0   0   0   1   1   0   1   0   0   1   1   0   1   0   1   1   0   1   1   0   1   0   1   0   0   0   0   0   1   0   0   1   1   1   0   1   1   0   0   1   0   0   1   0   0   1   1   0   0   0   0   0   0   1   1   1   0   1   0   0   1   0   0   0   1   1   1   0   0   0   1   0   0    &\  1  \cr
 0   0   0   0   0   1   0   0   1   1   1   0   1   1   0   0   1   0   0   1   0   0   1   1   0   0   0   0   0   0   1   1   1   0   1   0   0   1   0   0   0   1   1   1   0   0   0   1   0   0   0   0   0   0   0   1   0   1   1   0   0   0   1   1   1   1   0   1   0   0   0   0   1   1   1   1   1   1   1   1   0   0   1   0   0   0   0   1   0   1   0   0   1   1   1   1   1   0   1   0   1   0   1    &\  0  \cr
 0   0   0   0   0   0   1   0   1   1   0   0   0   1   1   1   1   0   1   0   0   0   0   1   1   1   1   1   1   1   1   0   0   1   0   0   0   0   1   0   1   0   0   1   1   1   1   1   0   1   0   1   0   1   0   1   1   1   0   0   0   0   0   1   1   0   0   0   1   0   1   0   1   1   0   0   1   1   0   0   1   0   1   1   1   1   1   1   0   1   1   1   1   0   0   1   1   0   1   1   1   0   1    &\  0  \cr
 0   0   0   0   0   0   0   1   0   1   1   0   0   0   1   1   1   1   0   1   0   0   0   0   1   1   1   1   1   1   1   1   0   0   1   0   0   0   0   1   0   1   0   0   1   1   1   1   1   0   1   0   1   0   1   0   1   1   1   0   0   0   0   0   1   1   0   0   0   1   0   1   0   1   1   0   0   1   1   0   0   1   0   1   1   1   1   1   1   0   1   1   1   1   0   0   1   1   0   1   1   1   0    &\  0  \cr
\end{align*}
} 

\medskip
\noindent $i=5$.

{\tiny
\begin{align*}
 1   0   0   0   0   0   0   1   1   0   0   1   1   0   1   1   0   0   0   1   1   1   0   0   1   1   1   0   1   0   1   1   1   0   0   0   0   1   0   0   1   1   0   0   0   0   0   1   0   1   0   1   0   1   1   0   1   0   0   1   0   0   1   0   1   0   0   1   1   1   1   0   0   1   0   0   0   1   1   0   1   0   1   0   0   0   0   1   1   1   1   1   1   1   0   1   1   1   0   1   1   0   1    &\  1  \cr
 0   1   0   0   0   0   0   0   1   1   0   0   1   1   0   1   1   0   0   0   1   1   1   0   0   1   1   1   0   1   0   1   1   1   0   0   0   0   1   0   0   1   1   0   0   0   0   0   1   0   1   0   1   0   1   1   0   1   0   0   1   0   0   1   0   1   0   0   1   1   1   1   0   0   1   0   0   0   1   1   0   1   0   1   0   0   0   0   1   1   1   1   1   1   1   0   1   1   1   0   1   1   0    &\  0  \cr
 0   0   1   0   0   0   0   0   0   1   1   0   0   1   1   0   1   1   0   0   0   1   1   1   0   0   1   1   1   0   1   0   1   1   1   0   0   0   0   1   0   0   1   1   0   0   0   0   0   1   0   1   0   1   0   1   1   0   1   0   0   1   0   0   1   0   1   0   0   1   1   1   1   0   0   1   0   0   0   1   1   0   1   0   1   0   0   0   0   1   1   1   1   1   1   1   0   1   1   1   0   1   1    &\  1  \cr
 0   0   0   1   0   0   0   0   0   0   1   1   0   0   1   1   0   1   1   0   0   0   1   1   1   0   0   1   1   1   0   1   0   1   1   1   0   0   0   0   1   0   0   1   1   0   0   0   0   0   1   0   1   0   1   0   1   1   0   1   0   0   1   0   0   1   0   1   0   0   1   1   1   1   0   0   1   0   0   0   1   1   0   1   0   1   0   0   0   0   1   1   1   1   1   1   1   0   1   1   1   0   1    &\  1  \cr
 0   0   0   0   1   0   0   1   1   0   0   0   0   0   1   0   1   0   1   0   1   1   0   1   0   0   1   0   0   1   0   1   0   0   1   1   1   1   0   0   1   0   0   0   1   1   0   1   0   1   0   0   0   0   1   1   1   1   1   1   1   0   1   1   1   0   1   1   0   1   1   1   1   0   1   0   0   0   1   0   1   1   0   0   1   0   1   1   1   1   1   0   0   0   1   0   0   0   0   0   0   1   1    &\  1  \cr
 0   0   0   0   0   1   0   1   0   1   0   1   1   0   1   0   0   1   0   0   1   0   1   0   0   1   1   1   1   0   0   1   0   0   0   1   1   0   1   0   1   0   0   0   0   1   1   1   1   1   1   1   0   1   1   1   0   1   1   0   1   1   1   1   0   1   0   0   0   1   0   1   1   0   0   1   0   1   1   1   1   1   0   0   0   1   0   0   0   0   0   0   1   1   0   0   1   1   0   1   1   0   0    &\  1  \cr
 0   0   0   0   0   0   1   1   0   0   1   1   0   1   1   0   0   0   1   1   1   0   0   1   1   1   0   1   0   1   1   1   0   0   0   0   1   0   0   1   1   0   0   0   0   0   1   0   1   0   1   0   1   1   0   1   0   0   1   0   0   1   0   1   0   0   1   1   1   1   0   0   1   0   0   0   1   1   0   1   0   1   0   0   0   0   1   1   1   1   1   1   1   0   1   1   1   0   1   1   0   1   1    &\  0  \cr
\end{align*}
} 

\medskip
\noindent $i=6$.

{\tiny
\begin{align*}
 1   0   0   0   0   0   0   1   0   0   0   1   0   0   1   1   0   0   0   1   0   1   1   1   0   1   0   1   1   0   1   1   0   0   0   0   0   1   1   0   0   1   1   0   1   0   1   0   0   1   1   1   0   0   1   1   1   1   0   1   1   0   1   0   0   0   0   1   0   1   0   1   0   1   1   1   1   1   0   1   0   0   1   0   1   0   0   0   1   1   0   1   1   1   0   0   0   1   1   1   1   1   1    &\  0  \cr
 0   1   0   0   0   0   0   0   1   0   0   0   1   0   0   1   1   0   0   0   1   0   1   1   1   0   1   0   1   1   0   1   1   0   0   0   0   0   1   1   0   0   1   1   0   1   0   1   0   0   1   1   1   0   0   1   1   1   1   0   1   1   0   1   0   0   0   0   1   0   1   0   1   0   1   1   1   1   1   0   1   0   0   1   0   1   0   0   0   1   1   0   1   1   1   0   0   0   1   1   1   1   1    &\  0  \cr
 0   0   1   0   0   0   0   0   0   1   0   0   0   1   0   0   1   1   0   0   0   1   0   1   1   1   0   1   0   1   1   0   1   1   0   0   0   0   0   1   1   0   0   1   1   0   1   0   1   0   0   1   1   1   0   0   1   1   1   1   0   1   1   0   1   0   0   0   0   1   0   1   0   1   0   1   1   1   1   1   0   1   0   0   1   0   1   0   0   0   1   1   0   1   1   1   0   0   0   1   1   1   1    &\  1  \cr
 0   0   0   1   0   0   0   1   0   0   1   1   0   0   0   1   0   1   1   1   0   1   0   1   1   0   1   1   0   0   0   0   0   1   1   0   0   1   1   0   1   0   1   0   0   1   1   1   0   0   1   1   1   1   0   1   1   0   1   0   0   0   0   1   0   1   0   1   0   1   1   1   1   1   0   1   0   0   1   0   1   0   0   0   1   1   0   1   1   1   0   0   0   1   1   1   1   1   1   1   0   0   0    &\  0  \cr
 0   0   0   0   1   0   0   0   1   0   0   1   1   0   0   0   1   0   1   1   1   0   1   0   1   1   0   1   1   0   0   0   0   0   1   1   0   0   1   1   0   1   0   1   0   0   1   1   1   0   0   1   1   1   1   0   1   1   0   1   0   0   0   0   1   0   1   0   1   0   1   1   1   1   1   0   1   0   0   1   0   1   0   0   0   1   1   0   1   1   1   0   0   0   1   1   1   1   1   1   1   0   0    &\  1  \cr
 0   0   0   0   0   1   0   0   0   1   0   0   1   1   0   0   0   1   0   1   1   1   0   1   0   1   1   0   1   1   0   0   0   0   0   1   1   0   0   1   1   0   1   0   1   0   0   1   1   1   0   0   1   1   1   1   0   1   1   0   1   0   0   0   0   1   0   1   0   1   0   1   1   1   1   1   0   1   0   0   1   0   1   0   0   0   1   1   0   1   1   1   0   0   0   1   1   1   1   1   1   1   0    &\  0  \cr
 0   0   0   0   0   0   1   0   0   0   1   0   0   1   1   0   0   0   1   0   1   1   1   0   1   0   1   1   0   1   1   0   0   0   0   0   1   1   0   0   1   1   0   1   0   1   0   0   1   1   1   0   0   1   1   1   1   0   1   1   0   1   0   0   0   0   1   0   1   0   1   0   1   1   1   1   1   0   1   0   0   1   0   1   0   0   0   1   1   0   1   1   1   0   0   0   1   1   1   1   1   1   1    &\  0  \cr
\end{align*}
}

\medskip
\noindent $i=7$.

{\tiny
\begin{align*}
 1   0   0   0   1   1   1   1   0   1   0   1   1   0   0   1   0   0   0   1   1   1   1   0   1   0   1   1   0   0   1   0   0   0   1   1   1   1   0   1   0   1   1   0   0   1   0   0   0   1   1   1   1   0   1   0   1   1   0   0   1   0   0   0   1   1   1   1   0   1   0   1   1   0   0   1   0   0   0   1   1   1   1   0   1   0   1   1   0   0   1   0   0   0   1   1   1   1   0   1   0   1   1    &\  1  \cr
 0   1   0   0   0   1   1   1   1   0   1   0   1   1   0   0   1   0   0   0   1   1   1   1   0   1   0   1   1   0   0   1   0   0   0   1   1   1   1   0   1   0   1   1   0   0   1   0   0   0   1   1   1   1   0   1   0   1   1   0   0   1   0   0   0   1   1   1   1   0   1   0   1   1   0   0   1   0   0   0   1   1   1   1   0   1   0   1   1   0   0   1   0   0   0   1   1   1   1   0   1   0   1    &\  0  \cr
 0   0   1   0   0   0   1   1   1   1   0   1   0   1   1   0   0   1   0   0   0   1   1   1   1   0   1   0   1   1   0   0   1   0   0   0   1   1   1   1   0   1   0   1   1   0   0   1   0   0   0   1   1   1   1   0   1   0   1   1   0   0   1   0   0   0   1   1   1   1   0   1   0   1   1   0   0   1   0   0   0   1   1   1   1   0   1   0   1   1   0   0   1   0   0   0   1   1   1   1   0   1   0    &\  0  \cr
 0   0   0   1   1   1   1   0   1   0   1   1   0   0   1   0   0   0   1   1   1   1   0   1   0   1   1   0   0   1   0   0   0   1   1   1   1   0   1   0   1   1   0   0   1   0   0   0   1   1   1   1   0   1   0   1   1   0   0   1   0   0   0   1   1   1   1   0   1   0   1   1   0   0   1   0   0   0   1   1   1   1   0   1   0   1   1   0   0   1   0   0   0   1   1   1   1   0   1   0   1   1   0    &\  1  \cr
\end{align*}
} 

\vskip-2.5em
\noindent $i=8$.

{\tiny
\begin{align*}
 1   0   0   0   0   0   0   0   1   0   1   1   1   0   1   1   1   1   0   1   1   0   0   1   1   1   1   1   1   1   1   0   0   1   0   1   1   0   1   0   1   1   0   1   1   1   0   1   0   1   0   1   0   1   1   1   0   0   1   0   0   1   1   0   1   1   0   1   0   0   1   1   0   0   1   1   0   1   0   0   0   1   1   1   0   1   1   0   1   1   0   0   0   1   0   0   0   1   0   0   1   1   1    &\  1  \cr
 0   1   0   0   0   0   0   0   0   1   0   1   1   1   0   1   1   1   1   0   1   1   0   0   1   1   1   1   1   1   1   1   0   0   1   0   1   1   0   1   0   1   1   0   1   1   1   0   1   0   1   0   1   0   1   1   1   0   0   1   0   0   1   1   0   1   1   0   1   0   0   1   1   0   0   1   1   0   1   0   0   0   1   1   1   0   1   1   0   1   1   0   0   0   1   0   0   0   1   0   0   1   1    &\  0  \cr
 0   0   1   0   0   0   0   0   1   0   0   1   0   1   0   1   0   0   1   0   1   1   1   1   1   0   0   0   0   0   0   1   1   1   0   0   1   1   0   0   0   1   1   0   1   0   1   0   0   0   0   0   0   0   1   0   1   1   1   0   1   1   1   1   0   1   1   0   0   1   1   1   1   1   1   1   1   0   0   1   0   1   1   0   1   0   1   1   0   1   1   1   0   1   0   1   0   1   0   1   1   1   0    &\  0  \cr
 0   0   0   1   0   0   0   0   0   1   0   0   1   0   1   0   1   0   0   1   0   1   1   1   1   1   0   0   0   0   0   0   1   1   1   0   0   1   1   0   0   0   1   1   0   1   0   1   0   0   0   0   0   0   0   1   0   1   1   1   0   1   1   1   1   0   1   1   0   0   1   1   1   1   1   1   1   1   0   0   1   0   1   1   0   1   0   1   1   0   1   1   1   0   1   0   1   0   1   0   1   1   1    &\  1  \cr
 0   0   0   0   1   0   0   0   0   0   1   0   0   1   0   1   0   1   0   0   1   0   1   1   1   1   1   0   0   0   0   0   0   1   1   1   0   0   1   1   0   0   0   1   1   0   1   0   1   0   0   0   0   0   0   0   1   0   1   1   1   0   1   1   1   1   0   1   1   0   0   1   1   1   1   1   1   1   1   0   0   1   0   1   1   0   1   0   1   1   0   1   1   1   0   1   0   1   0   1   0   1   1    &\  0  \cr
 0   0   0   0   0   1   0   0   1   0   1   0   1   0   0   1   0   1   1   1   1   1   0   0   0   0   0   0   1   1   1   0   0   1   1   0   0   0   1   1   0   1   0   1   0   0   0   0   0   0   0   1   0   1   1   1   0   1   1   1   1   0   1   1   0   0   1   1   1   1   1   1   1   1   0   0   1   0   1   1   0   1   0   1   1   0   1   1   1   0   1   0   1   0   1   0   1   1   1   0   0   1   0    &\  0  \cr
 0   0   0   0   0   0   1   0   1   1   1   0   1   1   1   1   0   1   1   0   0   1   1   1   1   1   1   1   1   0   0   1   0   1   1   0   1   0   1   1   0   1   1   1   0   1   0   1   0   1   0   1   1   1   0   0   1   0   0   1   1   0   1   1   0   1   0   0   1   1   0   0   1   1   0   1   0   0   0   1   1   1   0   1   1   0   1   1   0   0   0   1   0   0   0   1   0   0   1   1   1   1   0    &\  0  \cr
 0   0   0   0   0   0   0   1   0   1   1   1   0   1   1   1   1   0   1   1   0   0   1   1   1   1   1   1   1   1   0   0   1   0   1   1   0   1   0   1   1   0   1   1   1   0   1   0   1   0   1   0   1   1   1   0   0   1   0   0   1   1   0   1   1   0   1   0   0   1   1   0   0   1   1   0   1   0   0   0   1   1   1   0   1   1   0   1   1   0   0   0   1   0   0   0   1   0   0   1   1   1   1    &\  0  \cr
\end{align*}
} 

\noindent $i=9$.

{\tiny
\begin{align*}
 1   0   0   0   0   0   1   1   0   1   0   0   1   1   0   0   1   0   0   1   0   1   0   0   0   0   0   1   1   0   1   0   0   1   1   0   0   1   0   0   1   0   1   0   0   0   0   0   1   1   0   1   0   0   1   1   0   0   1   0   0   1   0   1   0   0   0   0   0   1   1   0   1   0   0   1   1   0   0   1   0   0   1   0   1   0   0   0   0   0   1   1   0   1   0   0   1   1   0   0   1   0   0    &\  1  \cr
 0   1   0   0   0   0   0   1   1   0   1   0   0   1   1   0   0   1   0   0   1   0   1   0   0   0   0   0   1   1   0   1   0   0   1   1   0   0   1   0   0   1   0   1   0   0   0   0   0   1   1   0   1   0   0   1   1   0   0   1   0   0   1   0   1   0   0   0   0   0   1   1   0   1   0   0   1   1   0   0   1   0   0   1   0   1   0   0   0   0   0   1   1   0   1   0   0   1   1   0   0   1   0    &\  1  \cr
 0   0   1   0   0   0   1   1   1   0   0   1   1   1   1   1   1   0   1   1   0   0   0   1   0   0   0   1   1   1   0   0   1   1   1   1   1   1   0   1   1   0   0   0   1   0   0   0   1   1   1   0   0   1   1   1   1   1   1   0   1   1   0   0   0   1   0   0   0   1   1   1   0   0   1   1   1   1   1   1   0   1   1   0   0   0   1   0   0   0   1   1   1   0   0   1   1   1   1   1   1   0   1    &\  1  \cr
 0   0   0   1   0   0   0   1   1   1   0   0   1   1   1   1   1   1   0   1   1   0   0   0   1   0   0   0   1   1   1   0   0   1   1   1   1   1   1   0   1   1   0   0   0   1   0   0   0   1   1   1   0   0   1   1   1   1   1   1   0   1   1   0   0   0   1   0   0   0   1   1   1   0   0   1   1   1   1   1   1   0   1   1   0   0   0   1   0   0   0   1   1   1   0   0   1   1   1   1   1   1   0    &\  0  \cr
 0   0   0   0   1   0   1   1   1   0   1   0   1   0   1   1   0   1   1   1   1   0   0   0   0   1   0   1   1   1   0   1   0   1   0   1   1   0   1   1   1   1   0   0   0   0   1   0   1   1   1   0   1   0   1   0   1   1   0   1   1   1   1   0   0   0   0   1   0   1   1   1   0   1   0   1   0   1   1   0   1   1   1   1   0   0   0   0   1   0   1   1   1   0   1   0   1   0   1   1   0   1   1    &\  1  \cr
 0   0   0   0   0   1   1   0   1   0   0   1   1   0   0   1   0   0   1   0   1   0   0   0   0   0   1   1   0   1   0   0   1   1   0   0   1   0   0   1   0   1   0   0   0   0   0   1   1   0   1   0   0   1   1   0   0   1   0   0   1   0   1   0   0   0   0   0   1   1   0   1   0   0   1   1   0   0   1   0   0   1   0   1   0   0   0   0   0   1   1   0   1   0   0   1   1   0   0   1   0   0   1    &\  0  \cr
\end{align*}
} 

\noindent $i=10$.

{\tiny
\begin{align*}
 1   0   0   0   0   0   0   0   1   1   1   1   1   0   0   0   1   0   0   1   0   1   0   0   0   1   0   0   0   1   0   1   1   0   1   0   1   0   1   1   1   0   0   1   0   0   0   1   1   0   0   1   0   0   1   0   0   1   1   1   0   0   0   0   0   0   1   0   0   0   0   1   0   0   1   1   0   1   1   1   1   0   0   1   1   0   0   1   1   1   0   1   1   1   1   1   1   0   0   1   0   1   1    &\  0  \cr
 0   1   0   0   0   0   0   0   0   1   1   1   1   1   0   0   0   1   0   0   1   0   1   0   0   0   1   0   0   0   1   0   1   1   0   1   0   1   0   1   1   1   0   0   1   0   0   0   1   1   0   0   1   0   0   1   0   0   1   1   1   0   0   0   0   0   0   1   0   0   0   0   1   0   0   1   1   0   1   1   1   1   0   0   1   1   0   0   1   1   1   0   1   1   1   1   1   1   0   0   1   0   1    &\  0  \cr
 0   0   1   0   0   0   0   0   1   1   0   0   0   1   1   0   1   0   1   1   0   0   0   1   0   1   0   1   0   1   0   0   1   1   0   0   0   0   0   1   0   1   1   1   0   1   0   1   1   1   1   1   0   1   1   0   1   1   1   0   1   1   0   0   0   0   1   0   1   0   0   1   0   1   1   1   1   0   1   0   0   1   1   1   1   1   1   1   1   0   1   0   1   0   0   0   0   1   1   1   0   0   1    &\  1  \cr
 0   0   0   1   0   0   0   0   1   0   0   1   1   0   1   1   1   1   0   0   1   1   0   0   1   1   1   0   1   1   1   1   1   1   0   0   1   0   1   1   0   0   1   0   1   0   1   1   0   1   1   0   1   0   0   1   0   0   0   0   0   1   1   0   0   0   1   1   0   1   0   1   1   0   0   0   1   0   1   0   1   0   1   0   0   1   1   0   0   0   0   0   1   0   1   1   1   0   1   0   1   1   1    &\  0  \cr
 0   0   0   0   1   0   0   0   0   1   0   0   1   1   0   1   1   1   1   0   0   1   1   0   0   1   1   1   0   1   1   1   1   1   1   0   0   1   0   1   1   0   0   1   0   1   0   1   1   0   1   1   0   1   0   0   1   0   0   0   0   0   1   1   0   0   0   1   1   0   1   0   1   1   0   0   0   1   0   1   0   1   0   1   0   0   1   1   0   0   0   0   0   1   0   1   1   1   0   1   0   1   1    &\  1  \cr
 0   0   0   0   0   1   0   0   0   0   1   0   0   1   1   0   1   1   1   1   0   0   1   1   0   0   1   1   1   0   1   1   1   1   1   1   0   0   1   0   1   1   0   0   1   0   1   0   1   1   0   1   1   0   1   0   0   1   0   0   0   0   0   1   1   0   0   0   1   1   0   1   0   1   1   0   0   0   1   0   1   0   1   0   1   0   0   1   1   0   0   0   0   0   1   0   1   1   1   0   1   0   1    &\  0  \cr
 0   0   0   0   0   0   1   0   0   0   0   1   0   0   1   1   0   1   1   1   1   0   0   1   1   0   0   1   1   1   0   1   1   1   1   1   1   0   0   1   0   1   1   0   0   1   0   1   0   1   1   0   1   1   0   1   0   0   1   0   0   0   0   0   1   1   0   0   0   1   1   0   1   0   1   1   0   0   0   1   0   1   0   1   0   1   0   0   1   1   0   0   0   0   0   1   0   1   1   1   0   1   0    &\  0  \cr
 0   0   0   0   0   0   0   1   1   1   1   1   0   0   0   1   0   0   1   0   1   0   0   0   1   0   0   0   1   0   1   1   0   1   0   1   0   1   1   1   0   0   1   0   0   0   1   1   0   0   1   0   0   1   0   0   1   1   1   0   0   0   0   0   0   1   0   0   0   0   1   0   0   1   1   0   1   1   1   1   0   0   1   1   0   0   1   1   1   0   1   1   1   1   1   1   0   0   1   0   1   1   0    &\  0  \cr
\end{align*}
} 

\noindent $i=11$.

{\tiny
\begin{align*}
 1   0   0   0   0   0   1   0   0   0   0   1   1   0   0   0   1   0   1   0   0   1   1   1   1   0   1   0   0   0   1   1   1   0   0   1   0   0   1   0   1   1   0   1   1   1   0   1   1   0   0   1   1   0   1   0   1   0   1   1   1   1   1   1   0   0   0   0   0   1   0   0   0   0   1   1   0   0   0   1   0   1   0   0   1   1   1   1   0   1   0   0   0   1   1   1   0   0   1   0   0   1   0    &\  1  \cr
 0   1   0   0   0   0   1   1   0   0   0   1   0   1   0   0   1   1   1   1   0   1   0   0   0   1   1   1   0   0   1   0   0   1   0   1   1   0   1   1   1   0   1   1   0   0   1   1   0   1   0   1   0   1   1   1   1   1   1   0   0   0   0   0   1   0   0   0   0   1   1   0   0   0   1   0   1   0   0   1   1   1   1   0   1   0   0   0   1   1   1   0   0   1   0   0   1   0   1   1   0   1   1    &\  0  \cr
 0   0   1   0   0   0   0   1   1   0   0   0   1   0   1   0   0   1   1   1   1   0   1   0   0   0   1   1   1   0   0   1   0   0   1   0   1   1   0   1   1   1   0   1   1   0   0   1   1   0   1   0   1   0   1   1   1   1   1   1   0   0   0   0   0   1   0   0   0   0   1   1   0   0   0   1   0   1   0   0   1   1   1   1   0   1   0   0   0   1   1   1   0   0   1   0   0   1   0   1   1   0   1    &\  0  \cr
 0   0   0   1   0   0   0   0   1   1   0   0   0   1   0   1   0   0   1   1   1   1   0   1   0   0   0   1   1   1   0   0   1   0   0   1   0   1   1   0   1   1   1   0   1   1   0   0   1   1   0   1   0   1   0   1   1   1   1   1   1   0   0   0   0   0   1   0   0   0   0   1   1   0   0   0   1   0   1   0   0   1   1   1   1   0   1   0   0   0   1   1   1   0   0   1   0   0   1   0   1   1   0    &\  1  \cr
 0   0   0   0   1   0   0   0   0   1   1   0   0   0   1   0   1   0   0   1   1   1   1   0   1   0   0   0   1   1   1   0   0   1   0   0   1   0   1   1   0   1   1   1   0   1   1   0   0   1   1   0   1   0   1   0   1   1   1   1   1   1   0   0   0   0   0   1   0   0   0   0   1   1   0   0   0   1   0   1   0   0   1   1   1   1   0   1   0   0   0   1   1   1   0   0   1   0   0   1   0   1   1    &\  0  \cr
 0   0   0   0   0   1   0   0   0   0   1   1   0   0   0   1   0   1   0   0   1   1   1   1   0   1   0   0   0   1   1   1   0   0   1   0   0   1   0   1   1   0   1   1   1   0   1   1   0   0   1   1   0   1   0   1   0   1   1   1   1   1   1   0   0   0   0   0   1   0   0   0   0   1   1   0   0   0   1   0   1   0   0   1   1   1   1   0   1   0   0   0   1   1   1   0   0   1   0   0   1   0   1    &\  0  \cr
\end{align*}
} 

\noindent $i=12$.

{\tiny
\begin{align*}
 1   0   0   0   0   0   0   0   1   1   0   0   0   1   1   1   1   1   0   1   0   0   1   0   1   1   0   1   1   0   1   0   0   0   0   0   1   0   0   1   0   1   0   1   1   1   0   0   1   1   1   1   0   0   0   0   0   0   0   1   1   0   0   0   1   1   1   1   1   0   1   0   0   1   0   1   1   0   1   1   0   1   0   0   0   0   0   1   0   0   1   0   1   0   1   1   1   0   0   1   1   1   1    &\  0  \cr
 0   1   0   0   0   0   0   0   1   0   1   0   0   1   0   0   0   0   1   1   1   0   1   1   1   0   1   1   0   1   1   1   0   0   0   0   1   1   0   1   1   1   1   1   0   0   1   0   1   0   0   0   1   0   0   0   0   0   0   1   0   1   0   0   1   0   0   0   0   1   1   1   0   1   1   1   0   1   1   0   1   1   1   0   0   0   0   1   1   0   1   1   1   1   1   0   0   1   0   1   0   0   0    &\  1  \cr
 0   0   1   0   0   0   0   0   0   1   0   1   0   0   1   0   0   0   0   1   1   1   0   1   1   1   0   1   1   0   1   1   1   0   0   0   0   1   1   0   1   1   1   1   1   0   0   1   0   1   0   0   0   1   0   0   0   0   0   0   1   0   1   0   0   1   0   0   0   0   1   1   1   0   1   1   1   0   1   1   0   1   1   1   0   0   0   0   1   1   0   1   1   1   1   1   0   0   1   0   1   0   0    &\  0  \cr
 0   0   0   1   0   0   0   0   0   0   1   0   1   0   0   1   0   0   0   0   1   1   1   0   1   1   1   0   1   1   0   1   1   1   0   0   0   0   1   1   0   1   1   1   1   1   0   0   1   0   1   0   0   0   1   0   0   0   0   0   0   1   0   1   0   0   1   0   0   0   0   1   1   1   0   1   1   1   0   1   1   0   1   1   1   0   0   0   0   1   1   0   1   1   1   1   1   0   0   1   0   1   0    &\  1  \cr
 0   0   0   0   1   0   0   0   1   1   0   1   0   0   1   1   0   1   0   1   0   1   0   1   1   0   1   0   1   1   0   0   1   1   1   0   1   0   0   0   1   1   1   0   0   0   1   0   1   0   1   0   0   0   0   1   0   0   0   1   1   0   1   0   0   1   1   0   1   0   1   0   1   0   1   1   0   1   0   1   1   0   0   1   1   1   0   1   0   0   0   1   1   1   0   0   0   1   0   1   0   1   0    &\  0  \cr
 0   0   0   0   0   1   0   0   1   0   1   0   1   1   1   0   0   1   1   1   1   0   0   0   0   0   0   0   1   1   0   0   0   1   1   1   1   1   0   1   0   0   1   0   1   1   0   1   1   0   1   0   0   0   0   0   1   0   0   1   0   1   0   1   1   1   0   0   1   1   1   1   0   0   0   0   0   0   0   1   1   0   0   0   1   1   1   1   1   0   1   0   0   1   0   1   1   0   1   1   0   1   0    &\  1  \cr
 0   0   0   0   0   0   1   0   1   0   0   1   0   0   0   0   1   1   1   0   1   1   1   0   1   1   0   1   1   1   0   0   0   0   1   1   0   1   1   1   1   1   0   0   1   0   1   0   0   0   1   0   0   0   0   0   0   1   0   1   0   0   1   0   0   0   0   1   1   1   0   1   1   1   0   1   1   0   1   1   1   0   0   0   0   1   1   0   1   1   1   1   1   0   0   1   0   1   0   0   0   1   0    &\  1  \cr
 0   0   0   0   0   0   0   1   1   0   0   0   1   1   1   1   1   0   1   0   0   1   0   1   1   0   1   1   0   1   0   0   0   0   0   1   0   0   1   0   1   0   1   1   1   0   0   1   1   1   1   0   0   0   0   0   0   0   1   1   0   0   0   1   1   1   1   1   0   1   0   0   1   0   1   1   0   1   1   0   1   0   0   0   0   0   1   0   0   1   0   1   0   1   1   1   0   0   1   1   1   1   0    &\  0  \cr
\end{align*}
} 

\vskip-2em
\noindent $i=13$.

{\tiny
\begin{align*}
 1   0   0   0   0   0   0   0   1   0   0   1   0   1   1   1   1   1   1   1   1   0   0   0   1   1   0   1   0   1   0   1   0   1   1   1   1   0   1   1   0   0   1   1   0   0   1   0   1   0   0   1   0   0   0   1   0   0   0   1   1   0   0   0   1   1   1   1   0   0   0   0   1   0   0   0   0   1   0   1   0   0   0   0   0   1   1   1   1   1   0   0   1   1   1   1   1   1   0   1   0   1   0    &\  0  \cr
 0   1   0   0   0   0   0   0   1   1   0   1   1   1   0   0   0   0   0   0   0   1   0   0   1   0   1   1   1   1   1   1   1   1   0   0   0   1   1   0   1   0   1   0   1   0   1   1   1   1   0   1   1   0   0   1   1   0   0   1   0   1   0   0   1   0   0   0   1   0   0   0   1   1   0   0   0   1   1   1   1   0   0   0   0   1   0   0   0   0   1   0   1   0   0   0   0   0   1   1   1   1   1    &\  1  \cr
 0   0   1   0   0   0   0   0   0   1   1   0   1   1   1   0   0   0   0   0   0   0   1   0   0   1   0   1   1   1   1   1   1   1   1   0   0   0   1   1   0   1   0   1   0   1   0   1   1   1   1   0   1   1   0   0   1   1   0   0   1   0   1   0   0   1   0   0   0   1   0   0   0   1   1   0   0   0   1   1   1   1   0   0   0   0   1   0   0   0   0   1   0   1   0   0   0   0   0   1   1   1   1    &\  0  \cr
 0   0   0   1   0   0   0   0   1   0   1   0   0   0   0   0   1   1   1   1   1   0   0   1   1   1   1   1   1   0   1   0   1   0   0   0   1   0   1   0   1   0   0   1   1   0   0   0   0   1   1   0   0   1   1   1   0   1   1   1   1   1   0   1   1   1   0   1   0   0   1   0   1   0   1   1   0   1   0   0   1   1   1   0   0   1   1   0   1   1   0   0   0   1   0   1   1   1   0   1   1   0   1    &\  0  \cr
 0   0   0   0   1   0   0   0   0   1   0   1   0   0   0   0   0   1   1   1   1   1   0   0   1   1   1   1   1   1   0   1   0   1   0   0   0   1   0   1   0   1   0   0   1   1   0   0   0   0   1   1   0   0   1   1   1   0   1   1   1   1   1   0   1   1   1   0   1   0   0   1   0   1   0   1   1   0   1   0   0   1   1   1   0   0   1   1   0   1   1   0   0   0   1   0   1   1   1   0   1   1   0    &\  1  \cr
 0   0   0   0   0   1   0   0   1   0   1   1   1   1   1   1   1   1   0   0   0   1   1   0   1   0   1   0   1   0   1   1   1   1   0   1   1   0   0   1   1   0   0   1   0   1   0   0   1   0   0   0   1   0   0   0   1   1   0   0   0   1   1   1   1   0   0   0   0   1   0   0   0   0   1   0   1   0   0   0   0   0   1   1   1   1   1   0   0   1   1   1   1   1   1   0   1   0   1   0   0   0   1    &\  0  \cr
 0   0   0   0   0   0   1   0   0   1   0   1   1   1   1   1   1   1   1   0   0   0   1   1   0   1   0   1   0   1   0   1   1   1   1   0   1   1   0   0   1   1   0   0   1   0   1   0   0   1   0   0   0   1   0   0   0   1   1   0   0   0   1   1   1   1   0   0   0   0   1   0   0   0   0   1   0   1   0   0   0   0   0   1   1   1   1   1   0   0   1   1   1   1   1   1   0   1   0   1   0   0   0    &\  0  \cr
 0   0   0   0   0   0   0   1   0   0   1   0   1   1   1   1   1   1   1   1   0   0   0   1   1   0   1   0   1   0   1   0   1   1   1   1   0   1   1   0   0   1   1   0   0   1   0   1   0   0   1   0   0   0   1   0   0   0   1   1   0   0   0   1   1   1   1   0   0   0   0   1   0   0   0   0   1   0   1   0   0   0   0   0   1   1   1   1   1   0   0   1   1   1   1   1   1   0   1   0   1   0   0    &\  1  \cr
\end{align*}
}

\medskip
\noindent $i=14$.

{\tiny
\begin{align*}
 1   0   0   0   0   0   1   0   1   1   1   1   1   1   0   0   1   0   1   0   1   0   0   0   1   1   0   0   1   1   1   1   0   1   1   1   0   1   0   1   1   0   1   0   0   1   1   0   1   1   0   0   0   1   0   0   1   0   0   0   0   1   1   1   0   0   0   0   0   1   0   1   1   1   1   1   1   0   0   1   0   1   0   1   0   0   0   1   1   0   0   1   1   1   1   0   1   1   1   0   1   0   1    &\  0  \cr
 0   1   0   0   0   0   1   1   1   0   0   0   0   0   1   0   1   1   1   1   1   1   0   0   1   0   1   0   1   0   0   0   1   1   0   0   1   1   1   1   0   1   1   1   0   1   0   1   1   0   1   0   0   1   1   0   1   1   0   0   0   1   0   0   1   0   0   0   0   1   1   1   0   0   0   0   0   1   0   1   1   1   1   1   1   0   0   1   0   1   0   1   0   0   0   1   1   0   0   1   1   1   1    &\  0  \cr
 0   0   1   0   0   0   0   1   1   1   0   0   0   0   0   1   0   1   1   1   1   1   1   0   0   1   0   1   0   1   0   0   0   1   1   0   0   1   1   1   1   0   1   1   1   0   1   0   1   1   0   1   0   0   1   1   0   1   1   0   0   0   1   0   0   1   0   0   0   0   1   1   1   0   0   0   0   0   1   0   1   1   1   1   1   1   0   0   1   0   1   0   1   0   0   0   1   1   0   0   1   1   1    &\  0  \cr
 0   0   0   1   0   0   1   0   0   0   0   1   1   1   0   0   0   0   0   1   0   1   1   1   1   1   1   0   0   1   0   1   0   1   0   0   0   1   1   0   0   1   1   1   1   0   1   1   1   0   1   0   1   1   0   1   0   0   1   1   0   1   1   0   0   0   1   0   0   1   0   0   0   0   1   1   1   0   0   0   0   0   1   0   1   1   1   1   1   1   0   0   1   0   1   0   1   0   0   0   1   1   0    &\  0  \cr
 0   0   0   0   1   0   1   1   1   1   1   1   0   0   1   0   1   0   1   0   0   0   1   1   0   0   1   1   1   1   0   1   1   1   0   1   0   1   1   0   1   0   0   1   1   0   1   1   0   0   0   1   0   0   1   0   0   0   0   1   1   1   0   0   0   0   0   1   0   1   1   1   1   1   1   0   0   1   0   1   0   1   0   0   0   1   1   0   0   1   1   1   1   0   1   1   1   0   1   0   1   1   0    &\  1  \cr
 0   0   0   0   0   1   0   1   1   1   1   1   1   0   0   1   0   1   0   1   0   0   0   1   1   0   0   1   1   1   1   0   1   1   1   0   1   0   1   1   0   1   0   0   1   1   0   1   1   0   0   0   1   0   0   1   0   0   0   0   1   1   1   0   0   0   0   0   1   0   1   1   1   1   1   1   0   0   1   0   1   0   1   0   0   0   1   1   0   0   1   1   1   1   0   1   1   1   0   1   0   1   1    &\  1  \cr
\end{align*}
} 

\medskip
\noindent $i=15$.

{\tiny
\begin{align*}
 1   0   0   0   0   0   0   0   1   1   0   1   0   0   1   0   1   1   0   0   0   0   0   0   0   1   1   0   1   0   0   1   0   1   1   0   0   0   0   0   0   0   1   1   0   1   0   0   1   0   1   1   0   0   0   0   0   0   0   1   1   0   1   0   0   1   0   1   1   0   0   0   0   0   0   0   1   1   0   1   0   0   1   0   1   1   0   0   0   0   0   0   0   1   1   0   1   0   0   1   0   1   1    &\  0  \cr
 0   1   0   0   0   0   0   0   1   0   1   1   1   0   1   1   1   0   1   0   0   0   0   0   0   1   0   1   1   1   0   1   1   1   0   1   0   0   0   0   0   0   1   0   1   1   1   0   1   1   1   0   1   0   0   0   0   0   0   1   0   1   1   1   0   1   1   1   0   1   0   0   0   0   0   0   1   0   1   1   1   0   1   1   1   0   1   0   0   0   0   0   0   1   0   1   1   1   0   1   1   1   0    &\  1  \cr
 0   0   1   0   0   0   0   0   1   0   0   0   1   1   1   1   0   0   0   1   0   0   0   0   0   1   0   0   0   1   1   1   1   0   0   0   1   0   0   0   0   0   1   0   0   0   1   1   1   1   0   0   0   1   0   0   0   0   0   1   0   0   0   1   1   1   1   0   0   0   1   0   0   0   0   0   1   0   0   0   1   1   1   1   0   0   0   1   0   0   0   0   0   1   0   0   0   1   1   1   1   0   0    &\  0  \cr
 0   0   0   1   0   0   0   0   0   1   0   0   0   1   1   1   1   0   0   0   1   0   0   0   0   0   1   0   0   0   1   1   1   1   0   0   0   1   0   0   0   0   0   1   0   0   0   1   1   1   1   0   0   0   1   0   0   0   0   0   1   0   0   0   1   1   1   1   0   0   0   1   0   0   0   0   0   1   0   0   0   1   1   1   1   0   0   0   1   0   0   0   0   0   1   0   0   0   1   1   1   1   0    &\  1  \cr
 0   0   0   0   1   0   0   0   1   1   1   1   0   0   0   1   0   0   0   0   0   1   0   0   0   1   1   1   1   0   0   0   1   0   0   0   0   0   1   0   0   0   1   1   1   1   0   0   0   1   0   0   0   0   0   1   0   0   0   1   1   1   1   0   0   0   1   0   0   0   0   0   1   0   0   0   1   1   1   1   0   0   0   1   0   0   0   0   0   1   0   0   0   1   1   1   1   0   0   0   1   0   0    &\  1  \cr
 0   0   0   0   0   1   0   0   0   1   1   1   1   0   0   0   1   0   0   0   0   0   1   0   0   0   1   1   1   1   0   0   0   1   0   0   0   0   0   1   0   0   0   1   1   1   1   0   0   0   1   0   0   0   0   0   1   0   0   0   1   1   1   1   0   0   0   1   0   0   0   0   0   1   0   0   0   1   1   1   1   0   0   0   1   0   0   0   0   0   1   0   0   0   1   1   1   1   0   0   0   1   0    &\  0  \cr
 0   0   0   0   0   0   1   0   1   1   1   0   1   1   1   0   1   0   0   0   0   0   0   1   0   1   1   1   0   1   1   1   0   1   0   0   0   0   0   0   1   0   1   1   1   0   1   1   1   0   1   0   0   0   0   0   0   1   0   1   1   1   0   1   1   1   0   1   0   0   0   0   0   0   1   0   1   1   1   0   1   1   1   0   1   0   0   0   0   0   0   1   0   1   1   1   0   1   1   1   0   1   0    &\  0  \cr
 0   0   0   0   0   0   0   1   1   0   1   0   0   1   0   1   1   0   0   0   0   0   0   0   1   1   0   1   0   0   1   0   1   1   0   0   0   0   0   0   0   1   1   0   1   0   0   1   0   1   1   0   0   0   0   0   0   0   1   1   0   1   0   0   1   0   1   1   0   0   0   0   0   0   0   1   1   0   1   0   0   1   0   1   1   0   0   0   0   0   0   0   1   1   0   1   0   0   1   0   1   1   0    &\  0  \cr
\end{align*}
}


\end{document}